\documentclass[11pt]{amsart}
\usepackage{amsfonts}
\usepackage{amsmath}
\usepackage{amssymb}
\usepackage{amsthm}
\usepackage[margin=3truecm]{geometry}
\usepackage[all]{xy}

\newtheorem{Theorem}{Theorem}[section]
\newtheorem{Proposition}[Theorem]{Proposition}
\newtheorem{Lemma}[Theorem]{Lemma}
\newtheorem{Corollary}[Theorem]{Corollary}

\theoremstyle{remark}
\newtheorem{Remark}[Theorem]{Remark}

\theoremstyle{definition}
\newtheorem{Definition}[Theorem]{Definition}

\begin{document}
\bibliographystyle{plain}
\title{Sobolev-Lorentz spaces in the Euclidean setting and counterexamples}
\author[\c{S}. Costea]{\c{S}erban Costea}
\address{\c{S}. Costea\\
Department of Mathematics\\
University of Pisa\\
Largo Bruno Pontecorvo 5\\
I-56127 Pisa, Italy}
\email{secostea@@mail.dm.unipi.it}


\thanks{The author was partly supported by the Project PRIN 2010-11 "Calculus of Variations" and by the University of Pisa via grant PRA-2015-0017.}

\begin{abstract}
This paper studies the inclusions between different Sobolev-Lorentz spaces $W^{1,(p,q)}(\Omega)$ defined on open sets $\Omega \subset {\mathbf{R}^n},$ where $n \ge 1$ is an integer, $1<p<\infty$ and $1 \le q \le \infty.$ We prove that if $1 \le q<r \le \infty,$ then $W^{1,(p,q)}(\Omega)$ is strictly included in $W^{1,(p,r)}(\Omega).$

We show that although $H^{1,(p,\infty)}(\Omega) \subsetneq W^{1,(p,\infty)}(\Omega)$ where $\Omega \subset {\mathbf{R}}^n$ is open and $n \ge 1,$ there exists a partial converse. Namely, we show that if a function $u$ in $W^{1,(p,\infty)}(\Omega), n \ge 1$ is such that $u$ and its distributional gradient $\nabla u$ have absolutely continuous $(p,\infty)$-norm, then $u$ belongs to $H^{1,(p,\infty)}(\Omega)$ as well.

We also extend the Morrey embedding theorem to the Sobolev-Lorentz spaces $H_{0}^{1,(p,q)}(\Omega)$ with $1 \le n<p<\infty$ and $1 \le q \le \infty.$ Namely, we prove that the Sobolev-Lorentz spaces $H_{0}^{1,(p,q)}(\Omega)$ embed into the space of H\"{o}lder continuous functions on $\overline{\Omega}$ with exponent $1-\frac{n}{p}$ whenever $\Omega \subset {\mathbf{R}}^n$ is open, $1 \le n<p<\infty,$ and $1 \le q \le \infty.$

\end{abstract}

\maketitle
\section{Introduction}

 In this paper we study the Sobolev-Lorentz spaces in the Euclidean setting and the inclusions between them. This paper is motivated by the results obtained in my 2006 PhD thesis \cite{Cos0} and in my book \cite{Cos3}. There I studied the Sobolev-Lorentz spaces and the associated Sobolev-Lorentz capacities in the Euclidean setting for $n \ge 2.$ The restriction on $n$ there was due to the the fact that I studied the $n,q$-capacity for $n>1.$

 The Sobolev-Lorentz spaces have also been studied by Cianchi-Pick in \cite{CiPi1} and \cite{CiPi2}, by Kauhanen-Koskela-Mal{\'{y}} in \cite{KKM}, and by Mal{\'{y}}-Swanson-Ziemer in \cite{MSZ}.

 The classical Sobolev spaces were studied by Gilbarg-Trudinger in \cite{GT}, Maz'ja in \cite{Maz}, Evans in \cite{Eva}, Heinonen-Kilpel\"{a}inen-Martio in \cite{HKM}, and by Ziemer in \cite{Zie}.

 The Lorentz spaces were studied by Bennett-Sharpley in \cite{BS}, Hunt in \cite{Hun}, and by Stein-Weiss in \cite{SW}.

 The Newtonian Sobolev spaces in the metric setting were studied by Shanmugalingam in \cite{Sha1} and \cite{Sha2}. See also Heinonen \cite{Hei}. Costea-Miranda studied the Newtonian Lorentz Sobolev spaces and the corresponding global $p,q$-capacities in \cite{CosMir}.

 There are several other definitions of Sobolev-type spaces in the metric setting
 when $p=q$; see Haj{\l}asz \cite{Haj1}, \cite{Haj2}, Heinonen-Koskela \cite{HeK}, Cheeger \cite{Che}, and Franchi-Haj{\l}asz-Koskela \cite{FHK}. It has been shown that under reasonable hypotheses, the majority of these definitions yields the same space; see Franchi-Haj{\l}asz-Koskela \cite{FHK} and Shanmugalingam \cite{Sha1}.

 The Sobolev-Lorentz relative $p,q$-capacity was studied in the Euclidean setting by Costea (see \cite{Cos0}, \cite{Cos1} and \cite{Cos3}) and by Costea-Maz'ya \cite{CosMaz}.
 The Sobolev $p$-capacity was studied by Maz'ya \cite{Maz} and by
 Heinonen-Kilpel\"{a}inen-Martio \cite{HKM} in ${\mathbf{R}}^n$ and by J. Bj{\"o}rn \cite{Bjo}, Costea \cite{Cos2} and Kinnunen-Martio \cite{KiM1} and \cite{KiM2} in metric spaces.

 The Sobolev-Lorentz spaces can be also studied in the Euclidean setting for $n=1.$ We do it in this paper. Many of the results on Sobolev-Lorentz spaces that we obtained in \cite{Cos0} and \cite{Cos3} in dimension $n \ge 2$ were extended here to the case $n=1.$

 In Section \ref{Section Lorentz spaces} we start by presenting some of the basic properties of the Lorentz spaces $L^{p,q}(\Omega; {\mathbf{R}}^m),$ where $\Omega \subset {\mathbf{R}}^n$ is open, $n, m \ge 1$ are integers, $1<p<\infty$ and $1 \le q \le \infty.$

 It is known that $L^{p,q}((0, \Omega_n r^n)) \subsetneq L^{p,s}((0, \Omega_n r^n)).$ We see this in Theorem \ref{Lpr stricly included in Lps} by constructing a function $u$ in $L^{p,s}((0, \Omega_n r^n)) \setminus L^{p,q}((0, \Omega_n r^n)).$ Here $r>0,$ $n \ge 1$, $1<p<\infty$ and $1 \le q<s \le \infty.$

 This function $u$ is used in Theorem \ref{Lpr strictly included in Lps via grad of smooth fns} to construct a radial function $v$ that is smooth in the punctured ball $B^{*}(0,r)$ such that $|\nabla v|$ is in $L^{p,s}(B(0,r)) \setminus L^{p,q}(B(0,r)).$ Later it will be shown in Theorem \ref{W1pr strictly included in W1ps} that $v$ is in $W^{1,(p,s)}(B(0,r)) \setminus W^{1,(p,q)}(B(0,r)).$
 This shows that the inclusion $W^{1,(p,q)}(B(0,r)) \subset W^{1,(p,s)}(B(0,r))$ is strict whenever $r>0,$ $n \ge 1$, $1<p<\infty$ and $1 \le q<s \le \infty.$

 In Section \ref{Section Sobolev Lorentz spaces} we revisit many of the results from my PhD thesis \cite[Chapter V]{Cos0} and from my book \cite[Chapter 3]{Cos3} and we extend them to the case $n=1.$ We improve some of the old results from \cite[Chapter V]{Cos0} and from \cite[Chapter 3]{Cos3}.

 We also obtain some new results in this section. Among them we mention the case $q=\infty$ for Theorems \ref{H=W revisited} and \ref{H=H_0 revisited} (see the discussion below) as well as the strict inclusion $W^{1,(p,q)}(B(0,r)) \subsetneq W^{1,(p,s)}(B(0,r))$ that we discussed above. As before, $r>0,$ $n \ge 1$, $1<p<\infty$ and $1 \le q<s \le \infty$ (see Theorem \ref{W1pr strictly included in W1ps}).

 For $n \ge 2,$ we proved in Costea \cite{Cos0} and \cite{Cos3} (by using partition of unity and convolution) that $H^{1,(p,q)}(\Omega)=W^{1,(p,q)}(\Omega)$ whenever $1<p<\infty$ and $1 \le q<\infty.$
 The partition of unity and convolution technique used there is similar to the techniques used by
 Ziemer in \cite{Zie} and by Heinonen-Kilpel\"{a}inen-Martio in \cite{HKM}.

 We proved in \cite{Cos0} and \cite{Cos3} (for $n \ge 2$) that $H^{1,(p,\infty)}(\Omega) \subsetneq W^{1,(p,\infty)}(\Omega).$ Once we constructed a function $u \in W^{1,(p, \infty)}(\Omega)$ such that its distributional gradient $\nabla u$ did not have an absolutely continuous $(p,\infty)$-norm, we proved there that $u$ was not in $H^{1,(p,\infty)}(\Omega).$

 In Section \ref{Section Sobolev Lorentz spaces} of this paper, Proposition \ref{up in Lpinfty loc and in W1pinfty loc minus H1pinfty} and Theorem \ref{H1pinfty subsetneq W1pinfty} show that $H^{1,(p,\infty)}(\Omega) \subsetneq W^{1,(p,\infty)}(\Omega)$ for $n \ge 1.$ In this paper we also give a partial converse. Namely, we show in Theorem \ref{H=W revisited} that if a function $u$ in $W^{1,(p,q)}(\Omega), n \ge 1,$ $1 \le q \le \infty$ is such that $u$ and its distributional gradient $\nabla u$ have absolutely continuous $(p,q)$-norm, then $u$ belongs to $H^{1,(p,q)}(\Omega)$ as well. This result is new for $q=\infty$ and $n \ge 1$ and improves a result from \cite{Cos0} and \cite{Cos3}, proved there for $n \ge 2$ and $1 \le q<\infty.$ We proved this result via a partition of unity and convolution argument, because convolution and partition of unity work well on functions $u$ that have absolutely continuous $(p,q)$-norm along with their distributional gradients $\nabla u.$

 In Theorem \ref{H=H_0 revisited} we show that if a function $u$ in $W^{1,(p,q)}({\mathbf{R}}^n), n \ge 1$ is such that $u$ and its distributional gradient $\nabla u$ have absolutely continuous $(p,q)$-norm, $1 \le q \le \infty,$ then $u$ belongs to $H_{0}^{1,(p,q)}({\mathbf{R}}^n)$ as well. This result is new when $q=\infty$ and $n \ge 1$ and improves a result from \cite{Cos0} and \cite{Cos3}, proved there for $n \ge 2$ and $1 \le q<\infty.$

 In Section \ref{Section Morrey embedding theorems} (among other things) we prove the Morrey embedding theorem for the Sobolev-Lorentz spaces $H_{0}^{1,(p,q)}(\Omega).$

 For $n=1,$ we prove in Theorem \ref{Holder 1/p' continuity for u in W1pq n equal 1} that if $\Omega \subset {\mathbf{R}}$ is an open interval, then $H_{0}^{1,(p,q)}(\Omega)$ and $W^{1,(p,q)}(\Omega)$ embed into the space of H\"{o}lder continuous functions in $\overline{\Omega}$ with exponent $1-\frac{1}{p}.$

 For $1<n<p<\infty,$ we prove in Theorem \ref{Morrey embedding 1<n<p} (among other things) that the spaces $H_{0}^{1,(p,q)}(\Omega)$ and $C_{0}(\Omega) \cap W^{1,(p,q)}(\Omega)$ embed into the space $C^{0, 1-\frac{n}{p}}(\overline{\Omega})$ of H\"{o}lder continuous functions in $\overline{\Omega}$ with exponent $1-\frac{n}{p}$ whenever $\Omega \subset {\mathbf{R}}^n$ is open
 and $1 \le q \le \infty.$

 Since we deal with functions in $H_{0}^{1,(p,q)}(\Omega)$ or in $C_{0}(\Omega) \cap W^{1,(p,q)}(\Omega)$ when $1<n<p<\infty$ and $1 \le q \le \infty,$ no regularity assumptions on $\partial \Omega$ are needed.

 When $1<n<p<\infty,$ we first prove the Morrey embedding $C_{0}(\Omega) \cap W^{1,(p,q)}(\Omega) \hookrightarrow C^{0, 1-\frac{n}{p}}(\overline{\Omega}).$ The embedding $H_{0}^{1,(p,q)}(\Omega) \hookrightarrow C^{0, 1-\frac{n}{p}}(\overline{\Omega})$ follows afterwards, after an approximation argument with functions from $C_{0}^{\infty}(\Omega).$ We also rely on the well-known Poincar\'{e} inequality in the Euclidean setting and we invoke the classical embedding for $1<n<s<p<\infty,$ proved by Evans in \cite{Eva} and by Gilbarg-Trudinger in \cite{GT}.

 \section{Notations}

 We recall the standard notation to be used throughout this paper.
 Throughout this paper, $C$ will denote a positive constant whose value
 is not necessarily the same at each occurrence; it may vary even within
 a line. $C(a,b, \ldots)$ is a constant that depends only on the parameters $a,b, \cdots.$

 Throughout this paper $\Omega$ will denote a nonempty open subset of ${\mathbf{R}}^n,$ while $dx=d m_n(x)$ will denote the Lebesgue $n$-measure in ${\mathbf{R}}^n,$ where $n \ge 1$ is integer.
 For $E \subset {\mathbf{R}}^n,$ the boundary, the closure, and the complement of $E$ with
 respect to ${\mathbf{R}}^n$ will be denoted by $\partial E,$ $\overline{E},$ and
 $\mathbf{R}^n \setminus E,$ respectively, while $|E|=\int_{E} dx$
 will denote the measure of $E$ whenever $E$ is measurable; $\mbox{diam }E$ is the Euclidean
 diameter of $E$ and $E \subset \subset F$ means that $\overline{E}$ is a compact subset of $F.$

 Moreover, $B(a,r)= \{ x \in \mathbf{R}^n: |x-a|<r \}$
 is the open ball with center $a \in \mathbf{R}^n$ and radius $r>0,$
 $B^{*}(a,r)=\{ x \in \mathbf{R}^n: 0<|x-a|<r \}$ is the punctured
 open ball with center $a \in \mathbf{R}^n$ and radius $r>0,$
 while $\overline{B}(a,r)= \{ x \in \mathbf{R}^n: |x-a| \le r \}$ is the
 closed ball with center $a \in \mathbf{R}^n$ and radius $r>0.$

 For two sets $A, B \subset \mathbf{R}^n,$ we define $\mbox{dist}(A,B),$ the distance
 between $A$ and $B,$ by
 \begin{equation*}
 \mbox{dist}(A,B)=\inf_{a \in A, b \in B} |a-b|.
 \end{equation*}

 For $n \ge 1$ integer, $\Omega_n$ denotes the Lebesgue measure of the
 $n$-dimensional unit ball. (That is, $\Omega_n=|B(0,1)|$). For $n \ge 2$ integer,
 $\omega_{n-1}$ denotes the spherical measure of the $n-1$-dimensional sphere;
 thus, $\omega_{n-1}=n \Omega_n$ for every integer $n \ge 2.$

 For $\Omega \subset \mathbf{R}^n,$ $C(\Omega)$ is the set of all
 continuous functions $u: \Omega \rightarrow {\mathbf{R}},$ while
 $C(\overline{\Omega})$ is the set of all continuous functions
 $u: \overline{\Omega} \rightarrow {\mathbf{R}}.$ Moreover,
 for a measurable $u: \Omega \rightarrow {\mathbf{R}},$ $\mbox {supp } u$ is the
 smallest closed set such that $u$ vanishes outside $\mbox {supp } u.$
 We also define
 \begin{eqnarray*}
 C^{k}(\Omega)&=&\{ \varphi: \Omega \rightarrow \mathbf{R}: \mbox{ the $k$th-derivative of $\varphi$ is continuous} \}\\
 C_{0}^{k}(\Omega)&=&\{\varphi \in C^{k}(\Omega): \mbox{supp } \varphi \subset \subset \Omega \}\\
 C^{\infty}(\Omega)&=& \bigcap_{k=1}^{\infty} C^{k}(\Omega)\\
 C_{0}^{\infty}(\Omega)&=& \{ \varphi \in C^{\infty}(\Omega): \mbox{supp } \varphi \subset \subset \Omega\}.
 \end{eqnarray*}
 For a function $\varphi \in C^{\infty}(\Omega)$ we write
 \begin{equation*}
 \nabla \varphi=(\partial_1 \varphi, \partial_2 \varphi, \ldots, \partial_n \varphi)
 \end{equation*}
 for the gradient of $\varphi.$

 Let $f: \Omega \rightarrow {\mathbf{R}}$ be integrable. For $E \subset
 \Omega$ measurable with $0<|E|<\infty,$ we define
 $$f_{E}=\frac{1}{|E|} \int_{E} f dx.$$

 For a measurable vector-valued function $f=(f_1, \ldots, f_m):
 \Omega \rightarrow \mathbf{R}^m,$ we let
 \begin{equation*}\label{def abs val of a vector function}
 |f|=\sqrt{f_1^2+f_2^2+\ldots+f_m^2}.
 \end{equation*}
 $L^{\infty}(\Omega; \mathbf{R}^m)$ denotes the space of essentially
 bounded measurable functions $u: \Omega \rightarrow \mathbf{R}^m$
 with
 \begin{equation*}
 ||u||_{L^{\infty}(\Omega)}=\mbox{ess} \sup |u| <\infty.
 \end{equation*}

\section{Lorentz Spaces} \label{Section Lorentz spaces}

\subsection{Definitions and basic properties}
 Let $f:\Omega \rightarrow \mathbf{R}$ be a measurable function. We
 define $\lambda_{[f]},$ the \textit{distribution function} of $f$ as follows (see Bennett-Sharpley \cite[Definition II.1.1]{BS} and Stein-Weiss \cite[p.\ 57]{SW}):
 $$\lambda_{[f]}(t)=|\{x \in \Omega: |f(x)| > t \}|, \qquad t \ge 0.$$
 We define $f^{*},$  the \textit{nonincreasing rearrangement} of $f$ by
 $$f^{*}(t)=\inf\{v: \lambda_{[f]}(v) \le t \}, \quad t \ge 0.$$
 (See Bennett-Sharpley \cite[Definition II.1.5]{BS} and Stein-Weiss \cite[p.\ 189]{SW}).
 We notice that $f$ and $f^{*}$ have the same distribution function.
 Moreover, for every positive $\alpha$ we have $(|f|^{\alpha})^{*}=(|f|^{*})^{\alpha}$
 and if $|g|\le |f|$ a.e. on $\Omega,$ then $g^{*}\le f^{*}.$
 (See Bennett-Sharpley \cite[Proposition II.1.7]{BS}).
 We also define $f^{**}$, the \textit{maximal function} of $f^{*}$ by
 $$f^{**}(t)=m_{f^{*}}(t)=\frac{1}{t} \int_{0}^{t} f^{*}(s) ds, \quad t >0.$$
 (See Bennett-Sharpley \cite[Definition II.3.1]{BS} and Stein-Weiss \cite[p.\ 203]{SW}).

Throughout this paper, we denote by $p'$ the H\"{o}lder
conjugate of $p \in [1,\infty],$ that is
$$p'=\left\{ \begin{array}{ll}
 \infty & \mbox{if $p=1$}\\
 \frac{p}{p-1} & \mbox{if $1<p<\infty$}\\
 1 & \mbox{if $p=\infty$}.
 \end{array}
 \right.
$$

The \textit{Lorentz space} $L^{p,q}(\Omega),$ $1<p<\infty,$ $1\le
q\le \infty,$ is defined as follows:
$$L^{p,q}(\Omega)= \{f: \Omega \rightarrow \mathbf{R}: f \mbox { is measurable and }
||f||_{L^{p,q}(\Omega)}<\infty\},$$
where
$$||f||_{L^{p,q}(\Omega)}=||f||_{p,q}=\left\{ \begin{array}{lc}
\left( \int_{0}^{\infty} (t^{\frac{1}{p}}f^{*}(t))^q \, \frac{dt}{t}
\right)^{\frac{1}{q}} & 1 \le q < \infty \\
\sup_{t>0} t \lambda_{[f]}(t)^{\frac{1}{p}}=\sup_{s>0}
s^{\frac{1}{p}} f^{*}(s) & q=\infty.
\end{array}
\right.
$$
(See Bennett-Sharpley \cite[Definition IV.4.1]{BS} and Stein-Weiss \cite[p.\ 191]{SW}). If $1 \le
q\le p,$ then $||\cdot||_{L^{p,q}(\Omega)}$ already represents a
norm, but for $p < q \le \infty$ it represents a quasinorm that is
equivalent to the norm $||\cdot||_{L^{(p,q)}(\Omega)},$ where
$$||f||_{L^{(p,q)}(\Omega)}=||f||_{(p,q)}=\left\{ \begin{array}{lc}
\left( \int_{0}^{\infty} (t^{\frac{1}{p}}f^{**}(t))^q \, \frac{dt}{t} \right)^{\frac{1}{q}} & 1 \le q < \infty \\
\sup_{t>0} t^{\frac{1}{p}} f^{**}(t) & q=\infty.
\end{array}
\right.
$$
(See Bennett-Sharpley \cite[Definition IV.4.4]{BS}).

Namely, from Lemma IV.4.5 in Bennett-Sharpley \cite{BS}  we have that
$$||f||_{L^{p,q}(\Omega)} \le ||f||_{L^{(p,q)}(\Omega)} \le \frac {p}{p-1} ||f||_{L^{p,q}(\Omega)}$$
for every $1\le q \le \infty.$

For a measurable vector-valued function $f=(f_1,\ldots, f_m): \Omega
\rightarrow  \mathbf{R}^m$ we say that $f \in L^{p,q}(\Omega;
\mathbf{R}^m)$ if and only if $f_i \in L^{p,q}(\Omega)$ for
$i=1,2,\ldots, m,$ if and only if $|f| \in L^{p,q}(\Omega)$ and we
define
\begin{equation*}
||f||_{L^{p,q}(\Omega;\mathbf{R}^m)}=||\,|f|\,||_{L^{p,q}(\Omega)}.
\end{equation*}
Similarly
\begin{equation*} ||f||_{L^{(p,q)}(\Omega;
\mathbf{R}^m)}=||\,|f|\,||_{L^{(p,q)}(\Omega)}.
\end{equation*}
 Obviously, it follows from the real-valued case that
$$||f||_{L^{p,q}(\Omega; \mathbf{R}^m)} \le ||f||_{L^{(p,q)}(\Omega; \mathbf{R}^m)}
\le \frac {p}{p-1} ||f||_{L^{p,q}(\Omega; \mathbf{R}^m)}$$ for every
$1 \le q \le \infty,$ and like in the real-valued case,
$||\cdot||_{L^{p,q}(\Omega; \mathbf{R}^m)}$ is already a norm when
$1\le q \le p,$ while it is a quasinorm when $p<q\le \infty.$

It is known that $(L^{p,q}(\Omega; \mathbf{R}^m),
||\cdot||_{L^{p,q}(\Omega; \mathbf{R}^m)})$ is a Banach space for
$1\le q \le p,$ while $(L^{p,q}(\Omega; \mathbf{R}^m),
||\cdot||_{L^{(p,q)}(\Omega; \mathbf{R}^m)})$ is a Banach space for
$1<p< \infty,$ $1\le q \le \infty.$ These spaces are reflexive if
$1<q<\infty$ and the dual of $L^{p,q}\Omega; \mathbf{R}^m)$ is, up to equivalence
of norms, the space $L^{p',q'}(\Omega; \mathbf{R}^m)$ for $1 \le q<\infty.$
See Bennett-Sharpley \cite[Theorem IV.4.7, Corollaries I.4.3 and
IV.4.8]{BS}, Hunt \cite[p.\ 259-262]{Hun}, the definition of $L^{p,q}(\Omega; \mathbf{R}^m)$
and the discussion after Proposition \ref{characterization of fns with absolutely continuous norm}.

\subsection{Absolute continuity of the $(p,q)$-norm and reflexivity of the Lorentz spaces}

\begin{Definition}\label{defn of fns with absolutely continuous norm}
{\rm(See Bennett-Sharpley \cite[Definition I.3.1]{BS}).}
Let $n, m \ge 1$ be two integers, $1<p<\infty$ and $1 \le q \le \infty.$ Suppose $\Omega \subset {\mathbf{R}^n}$ is open.
Let $X=L^{p,q}(\Omega; {\mathbf{R}}^m).$
A function $f$ in $X$ is said to have
\textit{absolutely continuous norm} in $X$ if and only if $||f
\chi_{E_k}||_{X} \rightarrow 0$ for every sequence $E_k$ satisfying
$E_k \rightarrow \emptyset$ a.e.
\end{Definition}

The following proposition gives a characterization of functions with absolutely continuous norm
in $X=L^{p,q}(\Omega; {\mathbf{R}}^m).$

\begin{Proposition} \label{characterization of fns with absolutely continuous norm}
{\rm(See Bennett-Sharpley \cite[Proposition I.3.6]{BS}).}
A function $f$ in $X$ has absolutely continuous norm if and only if the following condition holds: whenever $f_k$ {\rm($k=1,2, \ldots$)}, and $g$ are measurable functions satisfying $|f_k| \le |f|$ for all $k$ and $f_k \rightarrow g$ a.e., then $||f_k-g||_{X} \rightarrow 0.$
\end{Proposition}

Let $X_a$ be the subspace of $X$ consisting of functions with absolutely continuous
norm and let $X_b$ be the closure in $X$ of the set of simple functions. It is known
that $X_a=X_b$ when $X=L^{p,q}(\Omega; \mathbf{R}^m)$ for $1<p<\infty,$
$1\le q\le \infty,$ and $m \ge 1$ integer. (See Bennett-Sharpley \cite[Theorem I.3.13]{BS}).
Moreover, we have $X_a=X_b=X$ when $X=L^{p,q}(\Omega; \mathbf{R}^m)$
for $1<p<\infty,$ $1\le q<\infty,$ and $m \ge 1$ integer.
(See Bennett-Sharpley \cite[Theorem IV.4.7 and Corollary IV.4.8]{BS}
and the definition of $L^{p,q}(\Omega; \mathbf{R}^m)$).

From Proposition \ref{function in Lpinfty but not in Lpq q finite} it follows
that $X_a \subsetneq X$ for $X=L^{p,\infty}(\Omega;
\mathbf{R}^m)$ whenever $m \ge 1$ is an integer.
Since $L^{p,\infty}(\Omega; \mathbf{R}^m)$ can be identified with
$(L^{p',1}(\Omega; \mathbf{R}^m))^{*}$ (see Bennett-Sharpley \cite[Corollary IV.4.8]{BS}
and the definition of $L^{p,q}(\Omega; \mathbf{R}^m)$), it
follows from Bennett-Sharpley \cite[Corollaries I.4.3, I.4.4, IV.4.8 and Theorem
IV.4.7]{BS} that neither $L^{p,1}(\Omega; {\mathbf{R}^m}),$ nor $L^{p,\infty}(\Omega; {\mathbf{R}^m})$
are reflexive whenever $1<p<\infty.$

\subsection{Strict inclusions between Lorentz spaces}

\vspace{2mm}

\begin{Remark} \label{relation between Lpr and Lps}
It is known
(see Bennett-Sharpley \cite[Proposition IV.4.2]{BS}) that for every $p \in
(1,\infty)$ and $1\le r<s\le \infty$ there exists a constant $C(p,r,s)>0$ such
that
\begin{equation}\label{relation between Lpr and Lps norm}
||f||_{L^{p,s}(\Omega)} \le C(p,r,s) ||f||_{L^{p,r}(\Omega)}
\end{equation}
for all measurable functions $f \in L^{p,r}(\Omega).$ In particular,
$L^{p,r}(\Omega) \subset L^{p,s}(\Omega).$ Like in the real-valued
case, it follows that
\begin{equation*}
||f||_{L^{p,s}(\Omega; \mathbf{R}^m)} \le C(p,r,s)
||f||_{L^{p,r}(\Omega; \mathbf{R}^m)}
\end{equation*}
for every $m \ge 1$ integer and for all measurable functions $f \in
L^{p,r}(\Omega; \mathbf{R}^m),$ where $C(p,r,s)$ is the constant
from (\ref{relation between Lpr and Lps norm}). In particular,
$$L^{p,r}(\Omega; \mathbf{R}^m) \subset
L^{p,s}(\Omega;\mathbf{R}^m) \mbox{ for every $m \ge 1$ integer.}$$
\end{Remark}

The above inclusion is strict. (See Ziemer \cite[p.\ 37, Exercise 1.7]{Zie}). Given an open ball $B(0,r) \subset {\mathbf{R}^n},$ where $n \ge 1$ integer, $r>0$ and $1 \le q_1<q_2 \le \infty,$ we construct next in Theorem \ref{Lpr stricly included in Lps} a function $u \in L^{p,q_2}(B(0,r); \mathbf{R}^m) \setminus L^{p,q_1}(B(0,r);\mathbf{R}^m).$ In addition, Theorem \ref{Lpr stricly included in Lps} will allow us to construct later a radial function $v$ that is smooth in a punctured ball $B^{*}(0,r)$ such that $|\nabla v|$ is in $L^{p,q_2}(B(0,r)) \setminus L^{p,q_1}(B(0,r)).$ It is enough to assume that $m=1$ when proving this strict inclusion.

\begin{Theorem} \label{Lpr stricly included in Lps}
Let $n \ge 1$ be an integer. Let $0 <\alpha \le 1$ and $r>0.$
Suppose $1<p<\infty$ and $1 \le q_1<q_2 \le \infty.$
We define $u_{r, \alpha, p}$ on $[0, \Omega_n r^n)$ by
\begin{equation} \label{defn of uralphap}
u_{r, \alpha, p}(t)=t^{-\frac{1}{p}} \ln^{-\alpha} \left( \frac{\Omega_n r^n e^{p \alpha}}{t} \right). 
\end{equation}

We also define
\begin{eqnarray}
\label{defn of uradraplhap}
& & u_{rad, r, \alpha, p}:[0, r) \rightarrow [0, \infty], u_{rad, r, \alpha, p}(t):=u_{r, \alpha, p}(\Omega_n t^n) \mbox{ and }\\
\label{defn of uralphanp}
& & u_{r,\alpha, n, p}: B(0,r) \subset {\mathbf{R}^n} \rightarrow [0, \infty], u_{r, \alpha, n, p}(x):=u_{rad, r, \alpha, p}(|x|).
\end{eqnarray}

Then

{\rm(i)} $u_{r, \alpha, p}$ is a decreasing function on $[0, \Omega_n r^n)$ and
\begin{equation}
\label{uralphanpstar=uralphapstar=uralphap}
u_{r, \alpha, n, p}^{*}(t)=u_{r, \alpha, p}^{*}(t)=u_{r, \alpha, p}(t) \mbox{ for all } t \in [0, \Omega_n r^n).
\end{equation}

{\rm(ii)} $u_{r, \alpha, n, p} \in L^{p, q_2}(B(0,r)) \setminus L^{p,q_1}(B(0,r))$ if $1\le q_1 \le \frac{1}{\alpha} < q_2 \le \infty.$

\end{Theorem}

\begin{proof} (i) Since $u_{r, \alpha, p}$ is defined on $[0, \Omega_n r^n)$, it follows that $u_{r, \alpha, p}^{*}(t)=0$ whenever $\Omega_n r^n \le t<\infty.$
Similarly, since $u_{r, \alpha, n, p}$ is defined on $B(0,r)$ and $|B(0,r)|=\Omega_n r^n,$ it follows that $u_{r, \alpha, n, p}^{*}(t)=0$ whenever $\Omega_n r^n \le t<\infty.$ Once we show that $u_{r, \alpha, p}$ is decreasing on $[0, \Omega_n r^n),$ the definition of $u_{r, \alpha, n, p}$ implies immediately that $u_{r, \alpha, n, p}$ and $u_{r, \alpha, p}$ have the same distribution function, proving claim (i).

We see that $u_{r, \alpha, p}$ is smooth and strictly positive on $(0, \Omega_n r^n).$ Moreover, it is easy to see that $\lim_{t \rightarrow 0} u_{r, \alpha, p}(t)=\infty.$ Thus, in order to show that $u_{r, \alpha, p}$ is decreasing on $[0, \Omega_n r^n),$ it is enough to show that $u_{r, \alpha, p}'(t)<0$ on $(0, \Omega_n r^n).$

For $t \in (0, \Omega_n r^n)$ we have
\begin{eqnarray*}
u_{r, \alpha, p}'(t)&=&-\frac{1}{p} t^{-1-\frac{1}{p}} \ln^{-\alpha} \left( \frac{\Omega_n r^n e^{p \alpha}}{t} \right)+\alpha
t^{-1-\frac{1}{p}} \ln^{-\alpha-1} \left( \frac{\Omega_n r^n e^{p \alpha}}{t} \right)\\
&=&t^{-1-\frac{1}{p}} \ln^{-\alpha-1} \left( \frac{\Omega_n r^n e^{p \alpha}}{t} \right) \left(\alpha-\frac{1}{p} \ln \left( \frac{\Omega_n r^n e^{p \alpha}}{t} \right)\right).
\end{eqnarray*}

We see that $$\ln \left( \frac{\Omega_n r^n e^{p \alpha}}{t} \right)>0, \mbox{ for all } t \in (0, \Omega_n r^n).$$

Thus, for $t \in (0, \Omega_n r^n)$ we have
\begin{eqnarray*}
u_{r, \alpha, p}'(t)<0 &\iff& \alpha-\frac{1}{p} \ln \left( \frac{\Omega_n r^n e^{p \alpha}}{t} \right)<0 \iff\\
\ln \left( \frac{\Omega_n r^n e^{p \alpha}}{t} \right)>p \alpha &\iff& \frac{\Omega_n r^n e^{p \alpha}}{t}>e^{p \alpha} \iff \Omega_n r^n>t.
\end{eqnarray*}

But the last inequality is obviously true for all $t \in (0, \Omega_n r^n).$ Thus, $u_{r,\alpha, p}'$ is strictly negative on $(0, \Omega_n r^n),$ which implies that $u_{r, \alpha, p}$ is strictly decreasing on $[0, \Omega_n r^n).$

The definition of $u_{r, \alpha, n, p}$ and the fact that $u_{r, \alpha, p}$ is continuous, strictly decreasing and strictly positive on
$(0, \Omega_n r^n)$ imply immediately that $u_{r, \alpha, n, p}$ and $u_{r, \alpha, p}$ have the same distribution function.
This yields (\ref{uralphanpstar=uralphapstar=uralphap}), finishing the proof of (i).

\vskip 2mm

(ii) We proved in part (i) that $u_{r, \alpha, n, p}^{*}(t)=u_{r, \alpha, p}^{*}(t)$ for all $t \ge 0.$ Thus,
$$||u_{r, \alpha, n, p}||_{L^{p,q}(B(0,r))}=||u_{r, \alpha, p}||_{L^{p,q}((0, \Omega_n r^n))}$$
for every $q$ in $[1,\infty].$

For $1<p<\infty$ and $1 \le q \le \infty$ we let $I_{r, \alpha, p, q}=||u_{r, \alpha, p}||_{L^{p,q}((0,\Omega_n r^n))}.$

Then via (\ref{uralphanpstar=uralphapstar=uralphap}) we have
\begin{eqnarray*}
I_{r, \alpha, p, q}&=&\sup_{0 \le t \le \Omega_n r^n} t^{\frac{1}{p}} u_{r, \alpha, p}^{*}(t)
= \sup_{0 \le t \le \Omega_n r^n} t^{\frac{1}{p}} u_{r, \alpha, p}(t) \\
&=& \sup_{0 \le t \le \Omega_n r^n} \ln^{-\alpha} \left(\frac{\Omega_n r^n e^{p \alpha}}{t} \right)=(p \alpha)^{-\alpha}
\end{eqnarray*}
for $q=\infty$
and
\begin{eqnarray*}
I_{r, \alpha, p, q}^{q}&=&\int_{0}^{\Omega_n r^n} \left(t^{\frac{1}{p}} u_{r, \alpha, p}^{*}(t) \right)^q \frac{dt}{t}
= \int_{0}^{\Omega_n r^n} \left(t^{\frac{1}{p}} u_{r, \alpha, p}(t) \right)^q \frac{dt}{t}\\
&=& \int_{0}^{\Omega_n r^n} \ln^{-q \alpha} \left( \frac{\Omega_n r^n e^{p \alpha}}{t} \right) \frac{dt}{t}
\end{eqnarray*}
for $1 \le q<\infty.$

For a given $q$ in $[1, \infty),$ the last integral in the above sequence is an improper one and converges if and only if $1-q \alpha<0$ if and only if $\frac{1}{\alpha}<q.$
An easy computation shows that the value of the convergent improper integral is

$$\frac{1}{-1+q \alpha} \ln^{1-q \alpha} \left( \frac{\Omega_n r^n e^{p \alpha}}{\Omega_n r^n}\right)=\frac{(p \alpha)^{1-q \alpha}}{-1+q \alpha}.$$

Thus, if $1 \le q_1 \le \frac{1}{\alpha} < q_2 \le \infty,$ we have
$$||u_{r, \alpha, n, p}||_{L^{p,q_2}(B(0,r))}<\infty=||u_{r, \alpha, n, p}||_{L^{p,q_1}(B(0,r))}.$$

Hence, we proved that $u_{r, \alpha, n, p} \in L^{p,q_2}(B(0,r)) \setminus L^{p,q_1}(B(0,r)).$ This shows that the inclusion
$L^{p,q_1}(B(0,r)) \subset L^{p,q_2}(B(0,r))$ is strict whenever $1<p<\infty$ and $1 \le q_1 < q_2 \le \infty.$ This finishes the proof of the theorem.

\end{proof}

Theorem \ref{Lpr stricly included in Lps} allows us to construct a radial function $v$ that is smooth in a punctured ball $B^{*}(0,r)$ such that $|\nabla v|$ is in $L^{p,q_2}(B(0,r)) \setminus L^{p,q_1}(B(0,r)).$ Here $r>0,$ $n \ge 1,$ $1<p<\infty$ and $1 \le q_1<q_2 \le \infty.$

\begin{Theorem} \label{Lpr strictly included in Lps via grad of smooth fns}
Let $n \ge 1$ be an integer. Let $0<\alpha \le 1$ and $r>0.$ Suppose $1<p<\infty$ and $1 \le q_1<q_2 \le \infty.$

We define
\begin{equation}
\label{defn of fradralphap}
f_{rad, r, \alpha, p}:[0,r) \rightarrow [0, \infty], f_{rad, r, \alpha, p}(t)=\int_{t}^{r} u_{rad, r, \alpha, p}(s) ds,
\end{equation}
where $u_{rad, r, \alpha, p}$ is the function defined in {\rm(\ref{defn of uradraplhap})}.
We also define
\begin{equation}
\label{defn of vralphanp}
v_{r, \alpha, n, p}:B(0,r) \rightarrow [0, \infty], v_{r, \alpha, n, p}(x):=f_{rad, r, \alpha, p}(|x|).
\end{equation}

Then

{\rm(i)} $v_{r, \alpha, n, p} \in C^{\infty}(B^{*}(0,r))$ and $$\nabla v_{r, \alpha, n, p}(x)=f_{rad, r, \alpha, p}'(|x|) \frac{x}{|x|} \mbox{ for all } x \in B^{*}(0,r).$$

{\rm(ii)} $|\nabla v_{r, \alpha, n, p}(x)|=u_{r, \alpha, n, p}(x)$ for all $x \in B^{*}(0,r),$
where $u_{r, \alpha, n, p}$ is the function defined in {\rm(\ref{defn of uralphanp})}.

{\rm(iii)} $\lim_{x \rightarrow y} v_{r, \alpha, n, p}(x)=0$ for all $y \in \partial B(0,r).$

{\rm(iv)} If $p>n,$ then $v_{r, \alpha, n, p}$ is continuous in $B(0,r).$

{\rm(v)} If $1<p \le n,$ then $v_{r, \alpha, n, p}$ is unbounded on $B(0,r).$

{\rm(vi)} $|\nabla v_{r, \alpha, n, p}| \in L^{p,q_2}(B(0,r)) \setminus L^{p,q_1}(B(0,r))$ if
$1 \le q_1 \le \frac{1}{\alpha} < q_2 \le \infty.$

\end{Theorem}

\begin{proof}

Since $u_{rad, r, \alpha, p}$ is smooth in $(0,r)$ and bounded near $t=r,$ it follows immediately from the definition of $f_{rad, r, \alpha, p}$ that $f_{rad, r, \alpha, p}$ is smooth in $(0,r),$ $\lim_{t \rightarrow r} f_{rad, r, \alpha, p}(t)=0$ and
$f_{rad, r, \alpha, p}'(t)=-u_{rad, r, \alpha, p}(t)$ for all $t \in (0,r).$ This and the definition of $v_{r, \alpha, n, p}$ and $u_{r, \alpha, n, p}$ yield the claims
(i), (ii) and (iii) immediately.

Moreover, since
$$\lim_{t \rightarrow 0} f_{rad, r, \alpha, p}'(t)=-\lim_{t \rightarrow 0} u_{rad, r, \alpha, p}(t)=-\infty,$$ it follows immediately via (i) and (ii) that $v$ is not in $C^{\infty}(B(0,r)),$ because $v$ does not have a gradient at $x=0 \in B(0,r).$

We proved in (i) that $v_{r, \alpha, n, p} \in C^{\infty}(B^{*}(0,r)).$ Thus, the function $v_{r, \alpha, n, p}$ is continuous in $B(0,r)$ if and only if it is continuous at $x=0 \in B(0,r)$ if and only if $f_{rad, r, \alpha, p}$ is continuous at $t=0 \in [0,r).$ But from the definition of $f_{rad, r, \alpha, p},$ we see that this function is continuous at $t=0 \in [0,r)$ if and only $f_{rad, r, \alpha, p}(0)<\infty.$
Therefore, $v_{r, \alpha, n, p}$ is continuous in $B(0,r)$ if and only if $f_{rad, r, \alpha, p}(0)<\infty.$

\vskip 2mm

We prove now claim (iv). The definition of $u_{rad, r, \alpha, p}$ easily implies that
$$u_{rad, r, \alpha, p}(s) \le (\Omega_n s^n)^{-\frac{1}{p}} \ln^{-\alpha} (e^{p \alpha})= (p \alpha)^{-\alpha} (\Omega_n s^n)^{-\frac{1}{p}} $$
for all $s \in (0,r).$

For $1 \le n<p<\infty,$ the definition of $f_{rad, r, \alpha, p},$ the finiteness of the improper Riemann integral $\int_{0}^{r} s^{-\frac{n}{p}} ds,$ and the Comparison Test for improper Riemann integrals imply immediately that
\begin{eqnarray*}
f_{rad, r, \alpha, p}(0)&=&\lim_{t \rightarrow 0} f_{rad, r, \alpha, p}(t) \le (p \alpha)^{-\alpha} \Omega_n^{-\frac{1}{p}}
\lim_{t \rightarrow 0} \int_{t}^{r} s^{-\frac{n}{p}} ds \\
&=& (p \alpha)^{-\alpha} \Omega_n^{-\frac{1}{p}} \left(1-\frac{n}{p}\right)^{-1} r^{1-\frac{n}{p}}<\infty.
\end{eqnarray*}
Thus, if $1 \le n<p<\infty$ we have $f_{rad, r, \alpha, p}(0)<\infty,$ which implies (via the above discussion on the boundedness of $f_{rad, r, \alpha, p}(0)$) that $v_{r, \alpha, n, p}$ is continuous in $B(0,r).$

(v) For $1<p \le n,$ we show that $f_{rad, r, \alpha, p}(0)=\infty.$ We treat the cases $1<p<n$ and $1<p=n$ separately.

Case I. We consider first the case $1<p<n.$
We begin by showing that there exists a constant $m=m_{r, \alpha, n, p}>0$ such that
$$t^{\frac{1}{n}-\frac{1}{p}} \ln^{-\alpha} \left( \frac{\Omega_n r^n e^{p\alpha}}{t} \right) \ge m \mbox{ for all } t \in (0, \Omega_n r^n),$$
which is equivalent to showing that
$$t^{-\frac{1}{p}} \ln^{-\alpha} \left( \frac{\Omega_n r^n e^{p\alpha}}{t} \right) \ge m t^{-\frac{1}{n}} \mbox{ for all } t \in (0, \Omega_n r^n),$$
which is equivalent to showing that
$$u_{rad, r, \alpha, p}(s) \ge m (\Omega_n s^n)^{-\frac{1}{n}} \mbox{ for all } s \in (0,r).$$

Once we show the existence of such $m$, it follows immediately via the Comparison Test for improper Riemann integrals and the definition
of $f_{rad, r, \alpha, p}$ that
\begin{eqnarray*}
f_{rad, r, \alpha, p}(0)&=&\lim_{t \rightarrow 0} f_{rad, r, \alpha, p}(t) = \lim_{t \rightarrow 0} \int_{t}^{r} u_{rad, r, \alpha, p}(s) ds\\
 &\ge& m \Omega_n^{-\frac{1}{n}} \lim_{t \rightarrow 0} \int_{t}^{r} s^{-1} ds=\infty.
\end{eqnarray*}

This would prove the unboundedness of $v_{r, \alpha, n, p}$ on $B(0,r)$ when $1<p<n.$

We let $p_1=\frac{np}{n-p}.$ Thus, $p_1>p$ and $\frac{1}{p_1}=\frac{1}{p}-\frac{1}{n}.$ We define $h$ on the interval $[0, \Omega_n r^n)$ by
\begin{equation} \label{defn of hralphanp}
h(t)=h_{r, \alpha, n, p}(t)=t^{-\frac{1}{p_1}} \ln^{-\alpha} \left( \frac{\Omega_n r^n e^{p \alpha}}{t} \right).
\end{equation}
We notice that $h$ is smooth and strictly positive on $(0, \Omega_n r^n).$ Moreover, it is easy to see that $\lim_{t \rightarrow 0} h(t)=\infty.$
We compute $h'$ on $(0, \Omega_n r^n)$ and we notice that
$$h'(t)=t^{-1-\frac{1}{p_1}} \ln^{-\alpha-1} \left( \frac{\Omega_n r^n e^{p \alpha}}{t} \right)
\left(\alpha-\frac{1}{p_1} \ln \left(\frac{\Omega_n r^n e^{p \alpha}}{t} \right)\right), t \in (0, \Omega_n r^n).$$
We see that $h'(t)=0$ if and only if $t=t_{crit}=\Omega_n r^n e^{p\alpha-p_1\alpha} \in (0, \Omega_n r^n).$
We notice that $h$ has an unique global minimum on $(0, \Omega_n r^n),$ at $t=t_{crit}.$
We define $m=m_{r, \alpha, n, p}:=h(t_{crit}).$ Then $m>0$ and $h(t) \ge m>0$ for all $t \in [0, \Omega_n r^n).$ This proves the existence of the desired constant $m$ and finishes the proof of Case I.

Case II. We consider now the case $1<p=n.$

We compute effectively $f_{rad, r, \alpha, n}$ by considering the cases $\alpha=1$ and $\alpha \in (0,1)$ separately.

We assume first that $\alpha=1.$
For every $t \in (0,r)$ we have
\begin{eqnarray*}
f_{rad, r, 1, n}(t)&=&\int_{t}^{r} (\Omega_n s^{n})^{-\frac{1}{n}} \ln^{-1} \left( \frac{\Omega_n r^n e^n}{\Omega_n s^n} \right) ds\\
&=& \Omega_n^{-\frac{1}{n}} \int_{t}^{r} s^{-1} \ln^{-1} \left( \frac{r^n e^n}{s^n} \right) ds\\
&=& \Omega_n^{-\frac{1}{n}} n^{-1} \int_{t}^{r} s^{-1} \ln^{-1} \left( \frac{re}{s} \right) ds\\
&=& \Omega_n^{-\frac{1}{n}} n^{-1} \ln \left(\ln\left( \frac{re}{t} \right)\right)=\Omega_n^{-\frac{1}{n}} n^{-1} \ln\left(1+\ln\left(\frac{r}{t}\right)\right).
\end{eqnarray*}
Thus,
$$v_{r, 1, n, n}(x)=\Omega_n^{-\frac{1}{n}} n^{-1} \ln\left(1+\ln\left(\frac{r}{|x|}\right)\right) \mbox{ for all } x \in B(0,r).$$
It is easy to see that $v_{r, 1, n, n}$ is unbounded on $B(0,r).$ This proves Case II when $1<p=n$ and $\alpha=1.$

We assume now that $\alpha \in (0,1).$
For every $t \in (0,r)$ we have
\begin{eqnarray*}
f_{rad, r, \alpha, n}(t)&=&\int_{t}^{r} (\Omega_n s^{n})^{-\frac{1}{n}} \ln^{-\alpha} \left( \frac{\Omega_n r^n e^{n\alpha}}{\Omega_n s^n} \right) ds\\
&=& \Omega_n^{-\frac{1}{n}} \int_{t}^{r} s^{-1} \ln^{-\alpha} \left( \frac{r^n e^{n\alpha}}{s^n} \right) ds\\
&=& \Omega_n^{-\frac{1}{n}} n^{-\alpha} \int_{t}^{r} s^{-1} \ln^{-\alpha} \left( \frac{re^{\alpha}}{s} \right) ds\\
&=& \Omega_n^{-\frac{1}{n}} n^{-\alpha} (1-\alpha)^{-1} \left( \ln^{1-\alpha} \left( \frac{re^{\alpha}}{t} \right)-\alpha^{1-\alpha}  \right).
\end{eqnarray*}
Thus,
$$v_{r, \alpha, n, n}(x)=\frac{1}{\Omega_n^{\frac{1}{n}}n^{\alpha} (1-\alpha)} \left( \ln^{1-\alpha} \left( \frac{re^{\alpha}}{|x|} \right)-\alpha^{1-\alpha}  \right) \mbox{ for all } x \in B(0,r).$$
It is easy to see that $v_{r, \alpha, n, n}$ is unbounded on $B(0,r).$ This proves Case II when $1<p=n$ and $\alpha \in (0,1).$ This finishes the proof of claim (v).

\vskip 2mm

We prove now (vi). From part (ii) we have $|\nabla v_{r, \alpha, n, p}(x)|=u_{r, \alpha, n, p}(x)$
for all $x \in B^{*}(0,r).$ The claim follows immediately from the choice of $\alpha$ and Theorem \ref{Lpr stricly included in Lps} (ii). This finishes the proof of the theorem.

\end{proof}

The following proposition shows that $L^{p,\infty}$ does not have an absolutely continuous
$(p,\infty)$-norm. Moreover, it exhibits a function $u \in L^{p, \infty}$ that does not have
absolutely continuous $(p,\infty)$-norm and is not in $L^{p,q}$ for any $q$ in $[1, \infty).$

\begin{Proposition} \label{function in Lpinfty but not in Lpq q finite}
Let $n \ge 1$ be an integer. Let $r>0$ and $1<p<\infty.$
We define $$u_r: B(0,r) \rightarrow [0, \infty], u_{r}(x)=|x|^{-\frac{n}{p}}, 0 \le |x|<r.$$
Then

\par {\rm{(i)}} $u_r \in L^{p, \infty}(B(0,r))$ and $||u_r||_{L^{p, \infty}(B(0,r))}=\Omega_n^{\frac{1}{p}}.$

\par {\rm{(ii)}} $u_r \notin L^{p,q}(B(0,r))$ for every $q \in [1, \infty).$

\par {\rm{(iii)}} $u_r$ does not have absolutely continuous $(p,\infty)$-norm.

\par {\rm{(iv)}} If $v:B(0,r) \rightarrow {\mathbf{R}}$ is a locally bounded Lebesgue measurable function on $B(0,r),$ then
$$||u_r-v||_{L^{p,\infty}(B(0,\alpha))} \ge ||u_r||_{L^{p,\infty}(B(0,r))}$$
for every $\alpha \in (0,r).$

\end{Proposition}

\begin{proof} We compute $u_r^{*},$ the nonincreasing rearrangement of $u_r.$
In order to do that, we first compute $\lambda_{[u_r]},$ the distribution function of $u_r.$ For every $t \in [0, \infty)$ we have
\begin{eqnarray*}
\lambda_{[u_r]}(t)&=&|\{ x \in B(0,r): |u_r(x)|>t \}|=|\{ x \in B(0,r): |x|^{-\frac{n}{p}}>t \}|\\
&=&|\{ x \in B(0,r): |x|<t^{-\frac{p}{n}} \}|=|B(0, t^{-\frac{p}{n}}) \cap B(0,r)|\\
&=&\min (\Omega_n t^{-p}, \Omega_n r^n).
\end{eqnarray*}

Thus,

\begin{eqnarray*}
u_r^{*}(t)=\left\{ \begin{array}{ll}
\left(\frac{\Omega_n}{t}\right)^{\frac{1}{p}} & \mbox{ if $t \in [0, \Omega_n r^n)$}\\
0 & \mbox{ if $t \in [\Omega_n r^n, \infty)$}.
\end{array}
\right.
\end{eqnarray*}

This implies immediately that
$$||u_r||_{L^{p,\infty}(B(0,r))}= \sup_{t \in [0, \Omega_n r^n)} t^{\frac{1}{p}} u_r^{*}(t)=\sup_{t \in [0, \Omega_n r^n]} t^{\frac{1}{p}}
(\Omega_n t^{-1})^{\frac{1}{p}}=\Omega_n^{\frac{1}{p}}$$
and
\begin{eqnarray*}
||u_r||_{L^{p,q}(B(0,r))}^{q}&=&\int_{0}^{\Omega_n r^n} (t^{\frac{1}{p}} u_r^{*}(t))^{q} \frac{dt}{t}\\
&=&\int_{0}^{\Omega_n r^n} \left(t^{\frac{1}{p}} (\Omega_n t^{-1})^{\frac{1}{p}} \right)^{q} \frac{dt}{t}\\
&=&\int_{0}^{\Omega_n r^n} \Omega_n^{\frac{q}{p}} \frac{dt}{t}=\infty
\end{eqnarray*}
for all $q$ in $[1, \infty).$
This proves (i) and (ii).

\vskip 2mm

(iii) We prove now that the function $u_r$ does not have an absolutely continuous $(p,\infty)$-norm.
Let $\alpha \in (0,r)$ be fixed. Let $u_{r,\alpha}: B(0,r) \rightarrow [0, \infty]$ be the restriction of $u_r$ to
$B(0, \alpha).$
By doing a computation very similar to the computation of $u_r^{\*},$ we have
\begin{eqnarray*}
u_{r,\alpha}^{*}(t)=\left\{ \begin{array}{ll}
\left(\frac{\Omega_n}{t}\right)^{\frac{1}{p}} & \mbox{ if $t \in [0, \Omega_n \alpha^n)$}\\
0 & \mbox{ if $t \in [\Omega_n \alpha^n, \infty).$}
\end{array}
\right.
\end{eqnarray*}
Thus,
\begin{equation} \label{ur has constant pinfty norm on B0alpha}
||u_{r,\alpha}||_{L^{p,\infty}(B(0,\alpha))}=||u_r||_{L^{p,\infty}(B(0,\alpha))} =||u_r||_{L^{p,\infty}(B(0,r))}=\Omega_n^{\frac{1}{p}}
\end{equation}
for every $\alpha \in (0,r).$
This shows that $u_r$ does not have an absolutely continuous $(p,\infty)$-norm. This proves (iii).

\vskip 2mm

We prove now (iv). Let $v:B(0,r) \rightarrow {\mathbf{R}}$ be a Lebesgue measurable function that is locally bounded on $B(0,r).$ (Any continuous function on $B(0,r)$ is such a function). Let $\alpha \in (0,r)$ and $\varepsilon \in (0,1)$ be fixed. Let $M_{\alpha}>0$ be chosen such that $|v(x)|<M_{\alpha}$ for all $x \in B(0, \alpha).$
We have
$$|u_{r}(x)-v(x)| \ge |u_{r}(x)|-|v(x)| \ge |u_{r}(x)|-M_{\alpha}$$
for all $x \in B(0,\alpha).$
We want to find $\alpha_{\varepsilon} \in (0, \alpha)$ such that $M_{\alpha}<\varepsilon |u_{r}(x)|$ for all $x \in B(0,\alpha_{\varepsilon}).$
We have
$$M_{\alpha}<\varepsilon |u_{r}(x)| \iff \frac{M_{\alpha}}{\varepsilon}<|x|^{-\frac{n}{p}} \iff |x|^{\frac{n}{p}}<\frac{\varepsilon}{M_{\alpha}} \iff |x|<\left(\frac{\varepsilon}{M_{\alpha}}\right)^{\frac{p}{n}}.$$
If we choose
$$\alpha_{\varepsilon}=\min\left(\alpha, \left(\frac{\varepsilon}{M_{\alpha}}\right)^{\frac{p}{n}}\right),$$
the above computation, the definition of $u_{r}$ and the fact that $|v|< M_{\alpha}$ on $B(0,\alpha)$
imply that $$|u_{r}(x)-v(x)| \ge (1-\varepsilon) |u_{r}(x)|$$
for all $x$ in $B(0, \alpha_{\varepsilon}).$
Thus, we have
\begin{eqnarray*}
||u_{r}-v||_{L^{p,\infty}(B(0, \alpha))} &\ge& ||u_{r}-v||_{L^{p,\infty}(B(0, \alpha_{\varepsilon}))} \ge (1-\varepsilon) ||u_{r}||_{L^{p,\infty}(B(0, \alpha_{\varepsilon}))}\\
&=&(1-\varepsilon) ||u_{r}||_{L^{p,\infty}(B(0,r))}.
\end{eqnarray*}
The inequalities in the above sequence are obvious; we use (\ref{ur has constant pinfty norm on B0alpha}) for the equality in the above sequence. By letting $\varepsilon \rightarrow 0,$ we obtain the desired conclusion for a fixed $\alpha \in (0,r).$
Thus, we proved claim (iv). This finishes the proof.

\end{proof}

\subsection{H\"{o}lder Inequalities for Lorentz Spaces}

Here we record the following generalized H\"{o}lder inequalities for Lorentz spaces, previously proved in \cite{Cos1} and/or in \cite{Cos3}, valid for all integers $n \ge 1.$
\begin{Theorem}\label{Holder for Lorentz}
{\rm(See Costea \cite[Theorem 2.3]{Cos1} and \cite[Theorem 2.2.1]{Cos3}).}
Let $\Omega \subset \mathbf{R}^n.$ Suppose $1<p<\infty$ and $1\le q \le \infty.$ If $f \in L^{p,q}(\Omega)$ and $g \in L^{p', q'}(\Omega),$ then
$$\int_{\Omega}|f(x)g(x)| dx \le \int_{0}^{\infty} f^{*}(s) g^{*}(s) ds \le ||f||_{L^{p,q}(\Omega)} ||g||_{L^{p',q'}(\Omega)}.$$
\end{Theorem}

We have the following generalized H\"{o}lder inequality for Lorentz
spaces, valid for all integers $n \ge 1.$

\begin{Theorem} \label{Holder for Lorentz with general exponents}
{\rm(See Costea \cite[Theorem 2.2.2]{Cos3}).}
Suppose $\Omega \subset \mathbf{R}^n$ has finite measure. Let
$1<p_1,p_2, p_3<\infty,$ $1\le q_1, q_2, q_3\le \infty$ be such that
\begin{equation*}
\frac{1}{p_1}=\frac{1}{p_2}+\frac{1}{p_3}
\end{equation*}
and either
\begin{equation*}
\frac{1}{q_1}=\frac{1}{q_2}+\frac{1}{q_3}
\end{equation*}
whenever $1\le q_1, q_2, q_3<\infty$ or $1\le q_1=q_2 \le
q_3=\infty$ or $1 \le q_1=q_3 \le q_2=\infty.$ Then
\begin{equation*}
||f||_{L^{p_1,q_1}(\Omega; \mathbf{R}^m)} \le
||f||_{L^{p_2,q_2}(\Omega; \mathbf{R}^m)} \,
||\chi_{\Omega}||_{L^{p_3,q_3}(\Omega)}.
\end{equation*}
\end{Theorem}

As an application of Theorem \ref{Holder for Lorentz with general
exponents} we have the following result, valid for all integers $n \ge 1.$
\begin{Corollary}\label{Coro Holder for Lorentz}
{\rm(See Costea \cite[Corollary 2.4]{Cos1} and \cite[Corollary 2.2.3]{Cos3}).}
Let $1<p<q\le \infty$ and $\varepsilon \in (0, p-1)$ be fixed.
Suppose $\Omega \subset \mathbf{R}^n$ has finite measure. Then
\begin{equation}\label{Coro Holder for Lorentz 1}
||f||_{L^{p-\varepsilon}(\Omega; \mathbf{R}^m)} \le
C(p,q,\varepsilon) \,|\Omega|^{\frac{\varepsilon}{p(p-\varepsilon)}}
||f||_{L^{p,q}(\Omega; \mathbf{R}^m)}
\end{equation}
for every integer $m \ge 1,$ where
$$C(p,q,\varepsilon)=\left\{ \begin{array}{cc}
\left(\frac{p(q-p+\varepsilon)}{q}\right)^{\frac{1}{p-\varepsilon}-\frac{1}{q}} \, \varepsilon^{\frac{1}{q}-\frac{1}{p-\varepsilon}},& p<q<\infty \\
p^{\frac{1}{p-\varepsilon}} \,
\varepsilon^{-\frac{1}{p-\varepsilon}}, & q=\infty.
\end{array}
\right.
$$
\end{Corollary}

For the following definition, see Bennett-Sharpley \cite[Definition IV.4.17]{BS}.
\begin{Definition} For every measurable function $f$ on $\mathbf{R}^n,$ $n \ge 2,$
the \textit{fractional integral $I_{1}f$} is defined by
\begin{equation*}
(I_{1}f)(x)=\int_{\mathbf{R}^n} \frac{f(y)}{|x-y|^{n-1}} \, dy.
\end{equation*}
\end{Definition}

We record here the Hardy-Littlewood-Sobolev theorem of fractional
integration. (See Bennett-Sharpley \cite[Theorem IV.4.18]{BS} and Costea \cite[Theorem 2.2.5]{Cos3}).
\begin{Theorem}{\bf Hardy-Littlewood-Sobolev theorem.} \label{Hardy-Littlewood-Sobolev}
Let $1<p<n$ and $1 \le q \le \infty.$ Then there exists a constant $C(n,p,q)>0$ such that
\begin{equation}
||I_{1}f||_{L^{\frac{np}{n-p},q}({\mathbf{R}}^n)} \le C(n,p,q)
||f||_{L^{p,q}({\mathbf{R}}^n)}
\end{equation}
whenever $f \in L^{p,q}({\mathbf{R}}^n).$
\end{Theorem}

\section{Sobolev-Lorentz Spaces} \label{Section Sobolev Lorentz spaces}

This section is based in part on Chapter V of my PhD thesis \cite{Cos0} and on Chapter 3 of my book \cite{Cos3}. We generalize and extend some of the results from \cite{Cos0} and \cite{Cos3} to the case $n=1.$

Among the new results in this section we mention the case $q=\infty$ for Theorems \ref{H=W revisited} and \ref{H=H_0 revisited} as well as the inclusion $W^{1,(p,q)}(\Omega) \subsetneq W^{1,(p,s)}(\Omega),$ where $\Omega \subset {\mathbf{R}}^n$ is open, $n \ge 1$ is an integer, $1<p<\infty$ and $1 \le q<s \le \infty.$

\subsection{The $H^{1, (p,q)}$ and $W^{1, (p,q)}$ Spaces}

For $1<p<\infty$ and $1\le q \le \infty$ we define the Sobolev-Lorentz space
$H^{1, (p,q)}(\Omega)$ as follows.
Let $r=\min(p,q).$ For a function $\phi \in
C^{\infty}(\Omega)$ we define its Sobolev-Lorentz $(p,q)$-norm by
$$||\phi||_{1, (p,q); \Omega}=\left(||\phi||_{L^{(p,q)}(\Omega)}^{r}
+||\nabla \phi||_{L^{(p,q)}(\Omega;
\mathbf{R}^n)}^{r}\right)^{\frac{1}{r}},$$ where, we recall, $\nabla
\phi=(\partial_1 \phi, \ldots,
\partial_n \phi)$ is the gradient of $\phi.$
Similarly we define the Sobolev-Lorentz $p,q$-quasinorm of $\phi$ by
$$||\phi||_{1, p,q; \Omega}=\left(||\phi||_{L^{p,q}(\Omega)}^{r}
+||\nabla \phi||_{L^{p,q}(\Omega;
\mathbf{R}^n)}^{r}\right)^{\frac{1}{r}},$$
Then $H^{1, (p,q)}(\Omega)$ is defined as the completion of
$$\{\phi \in C^{\infty}(\Omega): ||\phi||_{1, (p,q); \Omega} < \infty \}$$
with respect to the norm $||\cdot||_{1, (p,q); \Omega}.$
Throughout the paper we might also use $||\cdot||_{H^{1, (p,q)}(\Omega)}$ instead of $||\cdot||_{1, (p,q); \Omega}$ and $||\cdot||_{H^{1, p,q}(\Omega)}$ instead of $||\cdot||_{1, p,q; \Omega}.$

The Sobolev-Lorentz space $H_{0}^{1, (p,q)}(\Omega)$ is defined as the closure of
$C_{0}^{\infty}(\Omega)$ in $H^{1, (p,q)}(\Omega)$. The
Sobolev-Lorentz spaces $H_{0}^{1, (p,q)}(\Omega)$ and $H^{1,
(p,q)}(\Omega)$ can be both regarded as closed subspaces of
$L^{(p,q)}(\Omega) \times L^{(p,q)}(\Omega; \mathbf{R}^n).$  Since
$L^{(p,q)}(\Omega) \times L^{(p,q)}(\Omega; \mathbf{R}^n)$ is
reflexive when $1<q< \infty,$  it follows that both $H_{0}^{1,
(p,q)}(\Omega)$ and $H^{1, (p,q)}(\Omega)$ are reflexive Banach
spaces when $1<q<\infty$ and have absolutely continuous norm when
$1\le q<\infty.$ In particular, $u \in L^{(p,q)}(\Omega)$ and
$\nabla u \in L^{(p,q)}(\Omega; \mathbf{R}^n)$ have absolutely
continuous $(p,q)$-norm whenever $1<p<\infty$ and $1\le
q<\infty.$

For $q=1$ we have that $(L^{(p,1)}(\Omega))^{*} \times
(L^{(p,1)}(\Omega; \mathbf{R}^n))^{*}$ can be regarded as a subspace
of $(H^{1,(p,1)}(\Omega))^{*}$ and since $(L^{(p,1)}(\Omega))^{*}
\times (L^{(p,1)}(\Omega; \mathbf{R}^n))^{*}$ can be identified with
the non-reflexive space $L^{(p',\infty)}(\Omega) \times
L^{(p',\infty)}(\Omega; \mathbf{R}^n),$ it follows that
$H^{1,(p,1)}(\Omega)$ is non-reflexive and so is
$H_{0}^{1,(p,1)}(\Omega),$ since it is a closed subspace of
$H^{1,(p,1)}(\Omega).$ It will be proved later in Theorem \ref{H1pinfty subsetneq W1pinfty} that none of these two spaces is reflexive when $q=\infty.$

Next we record the following reflexivity result, valid for all $n \ge 1$ and for all $q \in (1,\infty).$

\begin{Theorem} \label{HKM93 Thm132}
\rm{(See Costea \cite[Theorem V.22]{Cos0} and \cite[Theorem 3.5.4]{Cos3}).}
Let $1<p,q<\infty.$ Suppose that $u_j$ is a bounded sequence in $H^{1,(p,q)}(\Omega)$ such that $u_j \rightarrow u$ pointwise almost everywhere in $\Omega.$ Then $u \in H^{1,(p,q)}(\Omega).$ Moreover, if $u_j \in H_{0}^{1,(p,q)}(\Omega)$ for all $j \ge 1,$ then $u \in H_{0}^{1, (p,q)}(\Omega).$
\end{Theorem}

The following theorem generalizes the Gagliardo-Nirenberg-Sobolev inequality to the Sobolev-Lorentz spaces $H_{0}^{1,(p,q)}(\Omega)$ for $1<p<n$ and $1 \le q \le \infty.$ It also presents a Sobolev-Poincar\'{e} inequality for the Sobolev-Lorentz spaces $H_{0}^{1,(p,q)}(\Omega)$ when $\Omega \subset {\mathbf{R}}^n$ is open and bounded, $1<p<\infty$ and $1 \le q \le \infty.$

\begin{Theorem} \label{Sobolev-Poincare for Sobolev-Lorentz}
{\bf Sobolev inequalities for Sobolev-Lorentz spaces.}

Let $\Omega \subset \mathbf{R}^n$ be an open set, where $n \ge 2$ is an integer. Suppose $1<p<\infty$ and $1\le q \le \infty.$

\par {\rm(i)} If $1<p<n,$ then there exists a constant $C(n,p,q)>0$ such that
$$||u||_{L^{\frac{np}{n-p}, q}(\Omega)} \le C(n,p,q) ||\nabla u||_{L^{p,q}(\Omega; {\mathbf{R}^n})}$$
for every $u \in H_{0}^{1,(p,q)}(\Omega).$

\par {\rm(ii)} {\rm(See Costea \cite[Theorem 3.1.1]{Cos3}).} If $\Omega$ is bounded, then there exists a constant $C(n,p,q)>0$ such that
\begin{equation} \label{Sobolev-Poincare for Sobolev-Lorentz 1}
||u||_{L^{p,q}(\Omega)} \le  C(n,p,q) \, |\Omega|^{\frac{1}{n}}
||\nabla u||_{L^{p,q}(\Omega; \mathbf{R}^n)}
\end{equation}
for every $u \in H_{0}^{1,(p,q)}(\Omega).$
\end{Theorem}

\begin{proof} We have that $H_{0}^{1,(p,q)}(\Omega)$ is the closure of $C_{0}^{\infty}(\Omega)$ in $H^{1,(p,q)}(\Omega).$ Thus, via Costea \cite[Corollary 2.7]{Cos1}, it is enough to prove
claims (i) and (ii) for functions $u \in C_{0}^{\infty}(\Omega).$

Let $u$ be in $C_{0}^{\infty}(\Omega).$ We extend the function $u$ by $0$ on ${\mathbf{R}}^n \setminus \Omega$ and we denote this extension by $u$ as well. Then $u$ is in $C_{0}^{\infty}({\mathbf{R}}^n)$ and $u$ is compactly supported in $\Omega.$ Via Gilbarg-Trudinger \cite[Lemma 7.14]{GT}, we have
$$|u(x)| \le \frac{1}{\omega_{n-1}} (I_1|\nabla u|)(x)$$
for every $x \in {\mathbf{R}}^n.$ By using this pointwise inequality together with the Hardy-Littlewood-Sobolev Theorem (see Theorem \ref{Hardy-Littlewood-Sobolev} and Bennett-Sharpley \cite[Theorem IV.4.18]{BS}) it follows immediately that claim (i) holds for all functions $u \in C_{0}^{\infty}(\Omega).$ This proves claim (i) via Costea \cite[Corollary 2.7]{Cos1}.

\vskip 2mm

We prove now claim (ii). We have to consider two cases, depending on whether $1<p<n$ or $n \le p<\infty.$

Case I. First we assume that $1<p<n.$ We notice that $p<\frac{np}{n-p}.$ Via Theorems \ref{Holder for Lorentz with general exponents} and \ref{Hardy-Littlewood-Sobolev} it follows from part (i) that
\begin{equation*}
||u||_{L^{p,q}(\Omega)} \le |\Omega|^{\frac{1}{n}} ||u||_{L^{\frac{np}{n-p},q}(\Omega)}
\le \frac{C(n,p,q)}{\omega_{n-1}} |\Omega|^{\frac{1}{n}} ||\nabla u||_{L^{p,q}(\Omega; {\mathbf{R}^n})}
\end{equation*}
for every $u \in C_{0}^{\infty}(\Omega),$ where $C(n,p,q)$ is the constant from Theorem \ref{Hardy-Littlewood-Sobolev}. This proves the claim (ii) for $1<p<n$ via Costea \cite[Corollary 2.7]{Cos1}.

Case II. We assume now that $1<n \le p<\infty.$ We choose $s \in (1,n)$ such that $p<\frac{ns}{n-s}.$
Via Theorems \ref{Holder for Lorentz with general exponents} and \ref{Hardy-Littlewood-Sobolev} it follows from part (i) that
\begin{eqnarray*}
||u||_{L^{p,q}(\Omega)} &\le& |\Omega|^{\frac{1}{p}-\frac{n-s}{ns}} ||u||_{L^{\frac{ns}{n-s},q}(\Omega)}\\
&\le& \frac{C(n,s,q)}{\omega_{n-1}} |\Omega|^{\frac{1}{p}-\frac{n-s}{ns}} ||\nabla u||_{L^{s,q}(\Omega; {\mathbf{R}^n})}\\
&\le& \frac{C(n,s,q)}{\omega_{n-1}} |\Omega|^{\frac{1}{n}} ||\nabla u||_{L^{p,q}(\Omega; {\mathbf{R}^n})}
\end{eqnarray*}
for every $u \in C_{0}^{\infty}(\Omega),$ where $C(n,s,q)$ is the constant from Theorem \ref{Hardy-Littlewood-Sobolev}. This proves the claim (ii) for $1<n \le p<\infty$ via Costea \cite[Corollary 2.7]{Cos1}. This finishes the proof of the theorem.

\end{proof}

We recall that for $1<p<\infty,$ $H^{1,p}(\Omega)$ is defined as the
closure of $C^{\infty}(\Omega)$ with respect to the $||\cdot||_{1,p;
\Omega}$-norm, where
\begin{equation*}
||\psi||_{1,p; \Omega}=\left(\int_{\Omega} |\psi(x)|^{p} dx +
\int_{\Omega} |\nabla \psi(x)|^{p} dx \right)^{\frac{1}{p}}
\end{equation*}
for every $\psi \in C^{\infty}(\Omega).$ We recall that
$H^{1,p}_{loc}(\Omega)$ is defined in the obvious manner: a
measurable function $u: \Omega \rightarrow \mathbf{R}$ is in
$H^{1,p}_{loc}(\Omega)$ if and only if $u$ is in
$H^{1,p}_{loc}(\Omega')$ for every open set $\Omega' \subset \subset
\Omega.$

Let $u \in L_{loc}^{1}(\Omega).$ For $i=1,\ldots,n$ a function $v
\in L_{loc}^{1}(\Omega)$ is called the \textit{$i$th weak partial
derivative of} $u$ and we denote $v= \partial_{i} u$ if
$$\int_{\Omega} \varphi \, v \,dx=-\int_{\Omega} \partial_{i}\varphi \, u \,dx $$
for all $\varphi \in C_{0}^{\infty}(\Omega).$ Recall that
\begin{equation*}
W^{1,p}(\Omega)=L^{p}(\Omega) \cap \{ u: \partial_{i} u \in
L^{p}(\Omega), \, i=1,\ldots,n \}.
\end{equation*}
The space $W^{1,p}(\Omega)$ is equipped with the norm
\begin{equation*}
||u||_{W^{1,p}(\Omega)}=||u||_{L^{p}(\Omega)}+ \sum_{i=1}^{n}
||\partial_{i} u||_{L^{p}(\Omega)},
\end{equation*}
which is clearly equivalent to
\begin{equation*}
 \left(||u||_{L^{p}(\Omega)}^p+ ||\nabla u||_{L^{p}(\Omega;
\mathbf{R}^n)}^p\right)^{\frac{1}{p}}.
\end{equation*}
Here $\nabla u$ is the distributional gradient of $u.$
We recall that $W^{1,p}(\Omega)=H^{1,p}(\Omega).$

We define the Sobolev-Lorentz space $W^{1,(p,q)}(\Omega)$ by
\begin{equation*}
W^{1,(p,q)}(\Omega)=L^{(p,q)}(\Omega) \cap \{ u: \partial_{i} u \in
L^{(p,q)}(\Omega), \, i=1, \ldots, n \}.
\end{equation*}
The space $W^{1,(p,q)}(\Omega)$ is equipped with the norm
\begin{equation*}
||u||_{W^{1, (p,q)}(\Omega)}=||u||_{L^{(p,q)}(\Omega)}+
\sum_{i=1}^{n} ||\partial_{i} u||_{L^{(p,q)}(\Omega)},
\end{equation*}
which is clearly equivalent to
\begin{equation*}
 \left(||u||_{L^{(p,q)}(\Omega)}^r+ ||\nabla u||_{L^{(p,q)}(\Omega;
\mathbf{R}^n)}^r\right)^{\frac{1}{r}},
\end{equation*}
where $r=\min(p,q).$ As earlier, it is easy to see that $W^{1,(p,q)}(\Omega)$ is a
reflexive Banach space when $1<q<\infty$ and a non-reflexive Banach space when $q=1.$
It will be proved later in Theorem \ref{H1pinfty subsetneq W1pinfty} that
$W^{1,(p,\infty)}(\Omega)$ is not reflexive.

The corresponding local space $H_{loc}^{1, (p,q)}(\Omega)$ is
defined in the obvious manner:  $u$ is in $H_{loc}^{1,
(p,q)}(\Omega)$ if and only if $u$ is in $H^{1, (p,q)}(\Omega')$ for
every open set $\Omega' \subset \subset \Omega.$

Similarly, the local space $W_{loc}^{1, (p,q)}(\Omega)$ is
defined as follows:  $u$ is in $W_{loc}^{1,
(p,q)}(\Omega)$ if and only if $u$ is in $W^{1, (p,q)}(\Omega')$ for
every open set $\Omega' \subset \subset \Omega.$

The following theorem shows, among other things, the relation between
$W^{1,(p,q)}(\Omega)$ and $H_{loc}^{1,s}(\Omega),$ where $1<s<p<\infty$ and $1\le q \le
\infty.$

\begin{Theorem} \label{W1pqloc included in H1sloc s<p}
Let $\Omega \subset {\mathbf{R}}^n$ be an open set, where $n \ge 1$ is an integer.
Let $1<s<p<\infty$ and $1\le q<r \le \infty.$

\par {\rm(i)} We have $W^{1,(p,q)}(\Omega) \subset H_{loc}^{1,s}(\Omega).$ Moreover,
if $\Omega$ has finite Lebesgue measure (in particular if $\Omega$
is bounded), then $W^{1,(p,q)}(\Omega) \subset H^{1,s}(\Omega).$

\par {\rm(ii)} If $\Omega$ is bounded, then $H_{0}^{1,(p,q)}(\Omega) \subset H_{0}^{1,s}(\Omega).$

\par {\rm(iii)} We have $H_{0}^{1,(p,q)}(\Omega) \subset H_{0}^{1,(p,r)}(\Omega),$ $H^{1,(p,q)}(\Omega) \subset H^{1,(p,r)}(\Omega),$ and $W^{1,(p,q)}(\Omega) \subset W^{1,(p,r)}(\Omega).$

\end{Theorem}

\begin{proof} For claim (i), see Costea \cite[Theorem V.2]{Cos0} and \cite[Theorem 3.1.2]{Cos3}.
From either of these two references we can copy almost verbatim the proof, valid also for $n=1.$

Claims (ii) and (iii) follow immediately from Remark \ref{relation between Lpr and Lps}, Corollary \ref{Coro Holder for Lorentz}, and the definition of the Sobolev-Lorentz spaces on $\Omega.$

\end{proof}

We record the following theorem, which shows that every Sobolev element $u$ in
$H_{loc}^{1,(p,q)}(\Omega)$ is a distribution.

\begin{Theorem} \label{H included in W}
{\rm(See Costea \cite[Theorem V.3]{Cos0} and \cite[Theorem 3.1.3]{Cos3}).}
Suppose $1<p<\infty$ and $1\le q \le \infty.$
Let $u$ be in $H_{loc}^{1,(p,q)}(\Omega).$ Then $u$ is a
distribution with distributional gradient $\nabla u \in
L_{loc}^{1}(\Omega; \mathbf{R}^n).$ Moreover, $u \in
L_{loc}^{(p,q)}(\Omega) \subset L_{loc}^{1}(\Omega)$ and

\begin{equation*}
\int_{\Omega} u \, \partial_{i} \varphi \, dx= -\int_{\Omega}
\partial_{i}u \, \varphi \, dx
\end{equation*} for all $\varphi \in C_{0}^{\infty}(\Omega)$ and $i=1, \ldots, n,$
where $\partial_{i}u$ is the $i$th coordinate of $\nabla u.$ In
particular, $H^{1,(p,q)}(\Omega) \subset W^{1,(p,q)}(\Omega).$
\end{Theorem}

\subsection{Regularization}

We need some basic properties of the Sobolev-Lorentz spaces. Before
proceeding we recall the usual regularization procedure.

Let $\eta \in C_{0}^{\infty}(B(0,1))$ be a \textit{mollifier}. This means that $\eta$ is a nonnegative function such that
$$\int_{\mathbf{R}^n} \eta(x) \, dx=1.$$
Without loss of generality we can assume that $\eta$ is a radial function. Next
we write
$$\eta_{\varepsilon}(x)=\varepsilon^{-n} \eta(\varepsilon^{-1}x), \: \varepsilon>0.$$
For the basic properties of a mollifier see Ziemer \cite[Theorems 1.6.1 and 2.1.3]{Zie}.
We summarize the properties of the convolution (valid for all integers $n \ge 1$) in the following
theorem.

\begin{Theorem}\label{properties of convolutions in Lorentz spaces}
{\rm(See Costea \cite[Theorem V.4]{Cos0} and \cite[Theorem 3.2.1]{Cos3}).}
For $v \in L_{loc}^{1}({\mathbf{R}}^n),$ the convolution
$$v_{\varepsilon}(x)=\eta_{\varepsilon}*v(x)= \int_{{\mathbf{R}}^n} \eta_{\varepsilon}(x-y) v(y) dy$$ enjoys
the following properties for every $\varepsilon>0$:

\par {\rm{(i)}} For every $p \in (1, \infty)$ and every $q \in [1,\infty],$
there exists a constant $C(p,q)>0$ such that
\begin{equation}\label{Marcinkiewicz for convolution}
||v_{\varepsilon}||_{L^{(p,q)}({\mathbf{R}}^n)} \le C(p,q)
||v||_{L^{(p,q)}({\mathbf{R}}^n)}.
\end{equation}

\par {\rm{(ii)}} For every $p \in (1,\infty),$ every $q \in [1,\infty]$
and every $v \in L^{(p,q)}({\mathbf{R}}^n)$ with absolutely continuous
$(p,q)$-norm, we have
\begin{equation}\label{v_eps converges to v in Lpq}
||v_{\varepsilon}-v||_{L^{(p,q)}({\mathbf{R}}^n)} \rightarrow 0
\end{equation}
as $\varepsilon \rightarrow 0.$
\end{Theorem}

Recall that a function $u: \Omega \rightarrow {\mathbf{R}}$ is \textit{Lipschitz}
on $\Omega \subset {\mathbf{R}}^n,$ if there is $L>0$ such that
$$|u(x)-u(y)| \le L |x-y|$$
for all $x, y \in \Omega.$ Moreover, $u$ is \textit{locally Lipschitz} on $\Omega$ if
$u$ is Lipshitz on each compact subset of $\Omega.$

It is well known that every locally Lipschitz function on
${\mathbf{R}}^n$ is differentiable; this is Rademacher's theorem
(see Federer \cite[Theorem 3.1.6]{Fed}).

\begin{Lemma} \label{normal gradient}
{\rm(See Costea \cite[Lemma V.5]{Cos0} and \cite[Theorem 3.2.2]{Cos3}).}
Suppose $1<p<\infty$ and $1\le q \le \infty.$
Let $u: \Omega \rightarrow \mathbf{R}$ be a locally Lipschitz
function.  Then $u \in H_{loc}^{1,(p,q)}(\Omega)$ and $\nabla
u=(\partial_1 u, \ldots, \partial_n u)$ is the usual
gradient of $u.$
\end{Lemma}

\subsection{Product rule, density results and strict inclusions for Sobolev-Lorentz spaces}

\begin{Proposition} \label{up in Lpinfty loc and in W1pinfty loc minus H1pinfty}

 Let $n \ge 1$ be an integer and let $1<p<\infty.$ Let $u_p: {\mathbf{R}}^n \rightarrow [-\infty, \infty],$
 $$u_p(x)=\left\{ \begin{array}{ll}
 \ln |x| & \mbox{ if $p=n>1$}\\
 |x|^{1-\frac{n}{p}} & \mbox { if $p \neq n.$}
 \end{array}
 \right.
 $$
 Then $u_p$ is in $W^{1,(p,\infty)}_{loc}({\mathbf{R}}^n) \setminus H^{1,(p,\infty)}_{loc}({\mathbf{R}}^n).$

 \end{Proposition}

 \begin{proof} Since from Theorem \ref{W1pqloc included in H1sloc s<p} we have $W^{1,(p,\infty)}_{loc}({\mathbf{R}}^n) \subset H^{1,s}_{loc}({\mathbf{R}}^n)$ for all $1<s<p,$
 we prove first that $u_p$ is in $H^{1,s}_{loc}({\mathbf{R}}^n)$ for all $1<s<p.$

We start by noticing that $u_p \in L^{p}_{loc}({\mathbf{R}}^n)$ for $p=n>1$ and that $u_{p}(x) \le r |x|^{-\frac{n}{p}}$ for all $x \in B(0,r)$ and for all $p \neq n.$ Thus, $u_{p}$ is in $L^{(p,\infty)}_{loc}({\mathbf{R}}^n)$ for all $p \in (1, \infty).$ Moreover, an easy computation
shows that
$$\lim_{r \rightarrow \infty} ||u_{p}||_{L^{p,\infty}(B(0,r))}=\infty$$ for all $p \in (1, \infty).$ Thus, $u_{p}$ is not in $L^{p,\infty}({\mathbf{R}}^n).$

We notice that $u_{p}$ is smooth in ${\mathbf{R}}^n \setminus \{0\}$ with
$$\nabla u_{p}(x)=C(n,p) \, x \, |x|^{-1-\frac{n}{p}}, x \neq 0,$$
where
\begin{equation} \label{Cnp is 1-n/p or 1}
C(n,p)=\left\{ \begin{array}{ll}
1 & \mbox{ if $p=n>1$}\\
1-\frac{n}{p} & \mbox{ if $p \neq n.$}
\end{array}
\right.
\end{equation}

Thus, $|\nabla u_{p}(x)|=|C(n,p)| \, |x|^{-\frac{n}{p}}, x \neq 0.$
By doing a computation similar to the one
in Proposition \ref{function in Lpinfty but not in Lpq q finite}, we have
$$|\nabla u_p|^{*}(t)=|C(n,p)| \left(\frac{\Omega_n}{t}\right)^{1/p}$$
for all $t \ge 0,$ where $C(n,p)$ is the constant from (\ref{Cnp is 1-n/p or 1}).

Thus, it follows immediately that
$|\nabla u_{p}|$ is in $L^{(p,\infty)}({\mathbf{R}}^n)$ and
$$
||\nabla u_{p}||_{L^{p,\infty}({\mathbf{R}}^n; {\mathbf{R}}^n)}=||\nabla u_{p}||_{L^{p,\infty}(B(0,r); {\mathbf{R}}^n)}=|C(n,p)| \, \Omega_n^{1/p}<||\nabla u_{p}||_{L^{p,q}(B(0,r); {\mathbf{R}}^n)}=\infty
$$
for every $r>0$ and every $1 \le q<\infty,$ where $C(n,p)$ is the above constant.

By invoking Proposition \ref{function in Lpinfty but not in Lpq q finite} (iv), we see that
$$||\nabla u_{p}-\nabla v||_{L^{p,\infty}(B(0,\alpha); {\mathbf{R}}^n)} \ge ||\nabla u_{p}||_{L^{p,\infty}(B(0,r); {\mathbf{R}}^n)}
=||\nabla u_{p}||_{L^{p,\infty}({\mathbf{R}}^n; {\mathbf{R}}^n)}=|C(n,p)| \, \Omega_n^{1/p}>0
$$
for every $v \in C^{\infty}({\mathbf{R}}^n)$ and every $0<\alpha<r<\infty,$ where $C(n,p)$ is the
constant from (\ref{Cnp is 1-n/p or 1}). This implies immediately that $u_{p}$ is not in $H_{loc}^{1,(p,\infty)}({\mathbf{R}^n})$
because $u_{p}$ cannot be approximated with smooth functions in the $H^{1,(p,\infty)}$ norm on open balls centered at the origin.

It is enough to prove that $u_{p}$ is in $W^{1,(p,\infty)}(B(0,r))$ and in $H^{1,s}(B(0,r))$ for all $r>0$ and for all $s \in (1,p).$ We can assume without loss of generality that $r>1.$ We fix such $s$ and $r.$

For every integer $k \ge 1$ we truncate the function $u_{p}$ on the set $B(0, \frac{1}{k+1})$ and we denote this truncation by $u_{p,k}.$ Specifically, for $p=n>1$ and $k \ge 1$ integer we define $u_{n,k}$ on ${\mathbf{R}^n}$ by
$$u_{n,k}(x)=\left\{\begin{array}{ll}
\ln \frac{1}{k+1} & \mbox{ if $0 \le |x| \le \frac{1}{k+1}$}\\
u_{n}(x)=\ln |x| & \mbox{ if $\frac{1}{k+1} \le |x|< \infty.$}
\end{array}
\right.
$$
For $p \neq n$ and $k \ge 1$ integer we define $u_{p,k}$ on ${\mathbf{R}}^n$ by
$$u_{p,k}(x)=\left\{\begin{array}{ll}
\left(\frac{1}{k+1}\right)^{1-\frac{n}{p}} & \mbox{ if $0 \le |x| \le \frac{1}{k+1}$}\\
u_{p}(x)=|x|^{1-\frac{n}{p}} & \mbox{ if $\frac{1}{k+1} \le |x|<\infty.$}
\end{array}
\right.
$$

We notice that $(u_{p,k})_{k \ge 1} \subset H^{1,(p, \infty)}(B(0,r))$ is a sequence of Lipschitz functions on $\overline{B}(0,r).$

Moreover, for every $k \ge 1$ we have $0 \le |u_{p,k}| \le |u_{p}|$ pointwise in ${\mathbf{R}}^n$ and $|\nabla u_{p,k}| \le |\nabla u_{p}|$ almost everywhere in ${\mathbf{R}}^n.$ Thus, the sequence $u_{p,k}$ is bounded in $H^{1,(p, \infty)}(B(0,r))$ and in $H^{1,s}(B(0,r))$ for all $1<s<p.$

This sequence converges to $u_{p}$ pointwise in ${\mathbf{R}}^n \setminus \{0\}.$
The aforementioned pointwise convergence on $B^{*}(0,r)$ together with the reflexivity argument from Heinonen-Kilpel\"{a}inen-Martio \cite[Theorem 1.32]{HKM}, valid for all integers $n \ge 1,$ shows that
$u_{p}$ is in $H^{1,s}(B(0,r)).$ Thus, we showed that $u_{p} \in H^{1,s}(B(0,r))$ for all $1<s<p$ and all $r>0.$

We proved that $u_{p}$ is in $H^{1,s}_{loc}({\mathbf{R}}^n)$ for all $s \in (1,p).$ The fact that $u_{p}$ is in $H^{1,s}_{loc}({\mathbf{R}}^n)$ for all $1<s<p$ coupled with the fact that $u_{p}$ is in
$L^{(p,\infty)}_{loc}({\mathbf{R}}^n)$ and $|\nabla u_{p}|$ is in $L^{(p,\infty)}({\mathbf{R}}^n)$ show that $u_{p}$ is indeed in $W^{1,(p,\infty)}_{loc}({\mathbf{R}}^n).$ This finishes the proof.

 \end{proof}

The following theorem shows, among other things, that $H^{1,(p,\infty)}(\Omega) \subsetneq W^{1,(p,\infty)}(\Omega);$ it also shows that the spaces $H_{0}^{1,(p,\infty)}(\Omega),$
$H^{1,(p, \infty)}(\Omega),$ and $W^{1,(p, \infty)}(\Omega)$ are not reflexive.

\begin{Theorem} \label{H1pinfty subsetneq W1pinfty}
Let $\Omega \subset {\mathbf{R}}^n$ be an open set, where $n \ge 1$ is an integer and let $y$ be a point in $\Omega.$ Suppose $1<p<\infty.$

\par {\rm(i)} We have $H^{1,(p,\infty)}(\Omega) \subsetneq W^{1,(p,\infty)}(\Omega) \cap H^{1,(p,\infty)}(\Omega \setminus \{y \}).$

\par {\rm(ii)} We have $H^{1,(p,q)}(\Omega \setminus \{y \}) \subsetneq H^{1,(p,\infty)}(\Omega \setminus \{y \})$ whenever $1 \le q<\infty.$

\par {\rm(iii)} The spaces $H_{0}^{1,(p,\infty)}(\Omega),$ $H^{1,(p,\infty)}(\Omega),$ and $W^{1,(p,\infty)}(\Omega)$ are not reflexive.
\end{Theorem}

\begin{proof} The inclusion in (i) follows immediately from Theorem \ref{H included in W} and
from the definition of the Sobolev-Lorentz spaces $H^{1,(p, q)};$ the inclusion in (ii) follows immediately from Theorem \ref{W1pqloc included in H1sloc s<p} (iii) and from the
definition of the Sobolev-Lorentz spaces $H^{1,(p, q)}.$

In order to prove the strict inclusions in (i) and (ii) and the non-reflexivity in (iii), we can assume without loss of generality that $\Omega$ is a bounded open set in ${\mathbf{R}^n}$ such that $y=0 \in \Omega.$ Furthermore, we can assume without loss of generality that $\Omega=B(0,r)$ with $r>1.$

We prove (i) and (ii). We define $u_{r,p}:B(0,r) \rightarrow [-\infty, \infty]$ by
\begin{equation*}
u_{r,p}(x)=\left\{ \begin{array}{ll}
\ln \frac{|x|}{r} &  0 \le |x|<r, \mbox{ if $p=n>1$}\\
|x|^{1-\frac{n}{p}}-r^{1-\frac{n}{p}} & 0 \le |x|<r, \mbox{ if $p\neq n.$}
\end{array}
\right.
\end{equation*}

Let $c(n,p,r)$ be a constant that depends on $n,p,r,$ defined by
\begin{equation} \label{defn of cnpr}
c(n,p,r)=\left\{ \begin{array}{ll}
\ln r & \mbox{ if $p=n>1$}\\
r^{1-\frac{n}{p}} & \mbox{ if $p\neq n.$}
\end{array}
\right.
\end{equation}

We notice that $u_{r,p}=u_p-c(n,p,r)$ on $B(0,r),$ where $u_p$ is the function from Proposition \ref{up in Lpinfty loc and in W1pinfty loc minus H1pinfty}, defined on ${\mathbf{R}}^n.$ Thus, from Proposition \ref{up in Lpinfty loc and in W1pinfty loc minus H1pinfty} it follows immediately that $u_{r,p}$ is in $W^{1,(p,\infty)}(B(0,r)) \setminus H^{1,(p,\infty)}(B(0,r))$ and that $u_{r,p}$ is not in $H^{1,(p,q)}(B^{*}(0,r))$ whenever $1 \le q<\infty.$

Moreover, by mimicking the argument from the proof of Proposition \ref{up in Lpinfty loc and in W1pinfty loc minus H1pinfty}, we have
$$||\nabla u_{r,p}-\nabla v||_{L^{p,\infty}(B(0,\alpha); {\mathbf{R}}^n)} \ge ||\nabla u_{r,p}||_{L^{p,\infty}(B(0,r); {\mathbf{R}}^n)}=|C(n,p)| \, \Omega_n^{1/p}$$
for every $v \in C^{\infty}(B(0,r))$ and every $\alpha \in (0,r),$ where $C(n,p)$ is the
constant from (\ref{Cnp is 1-n/p or 1}).

We notice that $u_{r,p}$ is smooth in $B^{*}(0,r).$ Since we saw that $u_{r,p}$ is in $W^{1,(p,\infty)}(B(0,r)),$ it follows immediately that $u_{r,p} \in H^{1,(p,\infty)}(B^{*}(0,r)).$
This finishes the proof of claims (i) and (ii).

\vskip 2mm

We prove now claim (iii). We modify slightly the reflexivity argument from Proposition \ref{up in Lpinfty loc and in W1pinfty loc minus H1pinfty}. For every integer $k \ge 1$ we define the functions $u_{r,p,k}$ on $B(0,r)$ by
$u_{r,p,k}(x)=u_{p,k}(x)-c(n,p,r), x \in B(0,r);$ here $c(n,p,r)$ is the constant from (\ref{defn of cnpr}) and $u_{p,k}$ are the functions from Proposition \ref{up in Lpinfty loc and in W1pinfty loc minus H1pinfty} (namely the truncations of $u_p$ on $B(0, \frac{1}{k+1})$).
Specifically, for $p=n>1$ and $k \ge 1$ integer we have
$$u_{r,n,k}(x)=\left\{\begin{array}{ll}
\ln \frac{1}{k+1}-\ln r & \mbox{ if $0 \le |x| \le \frac{1}{k+1}$}\\
u_{r,n}(x)=\ln \frac{|x|}{r} & \mbox{ if $\frac{1}{k+1} \le |x| < r.$}
\end{array}
\right.
$$
For $p \neq n$ and $k \ge 1$ integer we have
$$u_{r,p,k}(x)=\left\{\begin{array}{ll}
\left(\frac{1}{k+1}\right)^{1-\frac{n}{p}}-r^{1-\frac{n}{p}} & \mbox{ if $0 \le |x| \le \frac{1}{k+1}$}\\
u_{r,p}(x)=|x|^{1-\frac{n}{p}}-r^{1-\frac{n}{p}} & \mbox{ if $\frac{1}{k+1} \le |x| < r.$}
\end{array}
\right.
$$

We see that $(u_{r,p,k})_{k \ge 1} \subset H_{0}^{1,(p, \infty)}(B(0,r))$ is a sequence of Lipschitz functions on $B(0,r)$ that can be extended continuously by $0$ on $\partial B(0,r).$ Moreover, for every $k \ge 1$ we have $0 \le |u_{r,p,k}| \le |u_{r,p}|$ pointwise in $B(0,r)$ and $|\nabla u_{r,p,k}| \le |\nabla u_{r,p}|$ almost everywhere in $B(0,r).$

By using the argument from Proposition \ref{up in Lpinfty loc and in W1pinfty loc minus H1pinfty} (ii) with minor modifications, we see that the sequence $u_{r,p,k}$ is bounded in $H_{0}^{1,(p, \infty)}(B(0,r)),$ in $H_{0}^{1,s}(B(0,r))$ and also in $H^{1,s}(B(0,r))$ for all $1<s<p.$ Since this sequence converges to $u_{r,p}$ pointwise in $B^{*}(0,r)$ but $u_{r,p}$ is not in $H^{1,(p,\infty)}(B(0,r)),$ it follows that $H_{0}^{1,(p,\infty)}(B(0,r))$ and $H^{1,(p,\infty)}(B(0,r))$ are not reflexive spaces. Moreover, since both these spaces are closed subspaces of $W^{1,(p,\infty)}(B(0,r)),$ it follows that the space $W^{1,(p,\infty)}(B(0,r))$ is not reflexive. Thus, we proved claim (ii). This finishes the proof of the theorem.

\end{proof}

The following lemma shows, among other things, that the product
between a function $u$ in $W^{1,(p,q)}(\Omega)$ and a function
$\varphi$ in $C_{0}^{\infty}(\Omega)$ yields a function in
$H_{0}^{1,(p,q)}(\Omega)$ if $1<p<\infty$ and $1\le q \le \infty$
whenever $u$ and $\nabla u$ have absolutely continuous
$(p,q)$-norm.

\begin{Lemma} \label{Product Rule for W1pq}
{\rm(See Costea \cite[Lemma V.6]{Cos0} and \cite[Lemma 3.3.1]{Cos3}).}
Let $\Omega \subset {\mathbf{R}}^n$ be an open set, where $n \ge 1$ is an integer.
Suppose $1<p<\infty$ and $1\le q\le \infty.$ Suppose that $u \in W^{1,(p,q)}(\Omega)$ and that $\varphi \in C_{0}^{\infty}(\Omega).$ Then $u \varphi \in W^{1,(p,q)}(\Omega)$
and $\nabla(u \varphi)=u \nabla \varphi+ \varphi \nabla u.$
Moreover, $u \varphi \in H_{0}^{1,(p,q)}(\Omega)$ whenever $u$ and
$\nabla u$ have absolutely continuous $(p,q)$-norm.
\end{Lemma}

\begin{proof} We can assume without loss of generality that $\Omega$ is bounded. Let $s \in (1,p).$
Then from Theorem \ref{W1pqloc included in H1sloc s<p} we have $u \in H^{1,s}(\Omega),$ and hence from Evans \cite[p.\ 247 Theorem 1]{Eva} it follows that $u \varphi \in H^{1,s}(\Omega)$ and $\nabla (u \varphi)=u \nabla \varphi+ \varphi \nabla u.$ Since $u \varphi \in L^{(p,q)}(\Omega)$ and $u \nabla \varphi+ \varphi \nabla u \in L^{(p,q)}(\Omega; {\mathbf{R}}^n),$ it follows that $u \varphi \in W^{1,(p,q)}(\Omega).$

Now suppose that $u$ and $\nabla u$ have absolutely continuous $(p,q)$-norm. (This is always the case when $1 \le q < \infty$). We have to prove that $u \varphi \in H_{0}^{1,(p,q)}(\Omega).$

If we multiply $u$ with a function $\widetilde{\eta} \in C_{0}^{\infty}(\Omega),$ the first part of the proof shows that both $u \widetilde{\eta}$ and $\nabla(u \widetilde{\eta})$ have absolutely continuous norm whenever $u$ and $\nabla u$ have absolutely continuous $(p,q)$-norm. If the function $\widetilde{\eta}$ is chosen to be $1$ on $\mbox{supp } \varphi,$
then $\widetilde{\eta} \varphi=\varphi.$ This allows us to assume without generality that $u$ has compact support in $\Omega.$

Let $\eta \in C_{0}^{\infty}(B(0,1))$ be a mollifier. Let $j_0>0$ be an integer such that
$$j_0>(\mbox{dist}(\mbox{supp } u, {\mathbf{R}}^n \setminus \Omega))^{-1}.$$
For $j \ge j_0$ integer we define $u_j: \Omega \rightarrow {\mathbf{R}},$ $u_j(x)=(\eta_j*u)(x),$ where $\eta_j(x)=j^n \eta(jx).$
We notice that $(u_j)_{j \ge j_0} \subset C_{0}^{\infty}(\Omega).$ Moreover, since $\eta_j \in C_{0}^{\infty}(B(0, j^{-1}))$ are mollifiers and $u \in W^{1,(p,q)}(\Omega),$ it follows via Ziemer \cite[Theorem 1.6.1]{Zie} that $\partial_i u_j=(\partial_i \eta_j)*u=\eta_j*(\partial_i u)$ for all
$i=1, \ldots, n$ and for all integers $j \ge j_0.$ Since $u$ and $\nabla u$ have absolutely continuous $(p,q)$-norm, it follows via Theorem \ref{properties of convolutions in Lorentz spaces} that $u_j$ converges to $u$ in $H^{1,(p,q)}(\Omega)$ as $j \rightarrow \infty.$

This implies, via the first part of the proof that $u_j \varphi, j \ge j_0$ is a sequence in $C_{0}^{\infty}(\Omega)$ that converges to
$u \varphi$ in $H^{1,(p,q)}(\Omega),$ which means that $u \varphi \in H_{0}^{1,(p,q)}(\Omega).$ This finishes the proof.

\end{proof}

\begin{Remark} \label{counterexample product rule q=infty}
We notice that the product $u \varphi$ does not necessarily belong to
$H_{0}^{1,(p,\infty)}(\Omega)$ when $u \in W^{1,(p,\infty)}(\Omega)$
and $\varphi \in C_{0}^{\infty}(\Omega)$ if $\nabla u$ does not have
absolutely continuous $(p,\infty)$-norm.

Indeed, let $0<\alpha<r<\infty$ and let $\Omega=B(0,r).$
Let $u_{r,p}$ be the function from Theorem \ref{H1pinfty subsetneq W1pinfty}.
Choose $\varphi_{r, \alpha}$ in $C_0^{\infty}(\Omega)$ such that $\varphi_{r, \alpha}=1$ in $B(0,\alpha).$ Then via Lemma \ref{Product Rule for W1pq} we have $u_{r, p} \varphi_{r, \alpha} \in W^{1,(p, \infty)}(\Omega).$ It is obvious that $u_{r, p} \varphi_{r, \alpha}=u_{r, p}$ in $B(0,\alpha)$ and hence via Theorem \ref{H1pinfty subsetneq W1pinfty} it follows that $u_{r, p} \varphi_{r, \alpha}$ does not belong to $H_{0}^{1,(p, \infty)}(\Omega).$

\end{Remark}

Now we prove that if $n \ge 1$ is an integer, $\Omega \subset {\mathbf{R}}^n$ is an open set and $u \in W^{1,(p,q)}(\Omega)$ is such that $u$ and $\nabla u$ have absolutely continuous $(p,q)$-norm, then $u \in H^{1,(p,q)}(\Omega).$ This result is new for $q=\infty$ and $n \ge 1.$ For $1 \le q<\infty$ it yields $H^{1,(p,q)}(\Omega)=W^{1, (p,q)}(\Omega),$ a result proved in Costea \cite[Theorem 3.3.4]{Cos3} and Costea \cite[Theorem V.9]{Cos0} for $n \ge 2.$ Thus, we generalize and improve the result obtained in Costea \cite[Theorem 3.3.4]{Cos3} and Costea \cite[Theorem V.9]{Cos0}.

\begin{Theorem} \label{H=W revisited}
Let $\Omega \subset {\mathbf{R}}^n$ be an open set, where $n \ge 1$ is an integer. Suppose $1<p<\infty$ and $1 \le q \le \infty.$ Suppose that $u \in W^{1,(p,q)}(\Omega).$ If $u$ and $\nabla u$ have absolutely
continuous $(p,q)$-norm, then $u \in H^{1,(p,q)}(\Omega).$ In particular,
$H^{1,(p,q)}(\Omega)=W^{1,(p,q)}(\Omega)$ if $1 \le q<\infty.$

\end{Theorem}

\begin{proof} Like in the proof of Ziemer \cite[Lemma 2.3.1]{Zie} and Costea \cite[Theorem 3.3.4]{Cos3}, 
we choose open sets $\Omega_0=\emptyset \subsetneq \Omega_j \subset \subset \Omega_{j+1}, j \ge 1$ such that $\cup_{j} \Omega_j=\Omega$ and a sequence of functions
$\psi_j, j\ge 1$ such that $\psi_j \in C_{0}^{\infty}(\Omega_{j+1} \setminus \overline{\Omega}_{j-1}),$ $0 \le \psi_j \le 1$ for every $j \ge 1$ and $\sum_j \psi_j \equiv 1$ in $\Omega.$

Let $\varepsilon>0$ be fixed. For every $j \ge 1,$ we have via Lemma \ref{Product Rule for W1pq} that $u \psi_j$ is in $H_{0}^{1,(p,q)}(\Omega).$
Moreover, since $\psi_j \in C_{0}^{\infty}(\Omega_{j+1} \setminus \overline{\Omega}_{j-1})$, we see that in fact $u \psi_j$ is in $H_{0}^{1,(p,q)}(\Omega_{j+1} \setminus \overline{\Omega}_{j-1})$ and thus, there exists $\varphi_{j}$ in $C_{0}^{\infty}(\Omega_{j+1} \setminus \overline{\Omega}_{j-1})$
such that
$$||\varphi_j- u \psi_j||_{1, (p,q); \Omega} \le ||\varphi_j- u \psi_j||_{L^{(p,q)}(\Omega)} + ||\nabla \varphi_j- \nabla (u \psi_j)||_{L^{(p,q)}(\Omega; {\mathbf{R}}^n)}
< \frac{\varepsilon}{2^{j}}$$
for all $j \ge 1$. If we define $\varphi \equiv \sum_{j \ge 1} \varphi_j,$ we see that $\varphi \in C^{\infty}(\Omega)$ because $\varphi$ can be written as a finite sum
of the functions $\varphi_i \in C_{0}^{\infty}(\Omega)$ on every bounded open set $U \subset \subset \Omega.$ Moreover,
$$||\varphi-u||_{1, (p,q); \Omega}=||\sum_{j \ge 1} (\varphi_j-u \psi_j)||_{1, (p,q); \Omega} \le \sum_{j \ge 1} ||\varphi_j-u \psi_j||_{1, (p,q); \Omega}< \sum_{j \ge 1} \frac{\varepsilon}{2^{j}}=\varepsilon.$$
This finishes the proof of the theorem.

\end{proof}

Now we prove that if $n \ge 1$ is an integer and $u \in W^{1,(p,q)}({\mathbf{R}}^n)$ is such that $u$ and $\nabla u$ have absolutely continuous $(p,q)$-norm, then $u \in H_{0}^{1,(p,q)}({\mathbf{R}}^n).$ This result is new for $q=\infty$ and $n \ge 1.$ For $1 \le q<\infty$ it yields $H^{1,(p,q)}({\mathbf{R}}^n)=H_{0}^{1, (p,q)}({\mathbf{R}}^n),$ a result proved in Costea \cite[Theorem 3.3.6]{Cos3} and Costea \cite[Theorem V.16]{Cos0} for $n \ge 2.$ Thus, we generalize and improve the result obtained in Costea \cite[Theorem 3.3.6]{Cos3} and Costea \cite[Theorem V.16]{Cos0}.

\begin{Theorem} \label{H=H_0 revisited}
Suppose $1<p<\infty$ and $1 \le q \le \infty.$ Suppose that $u \in W^{1,(p,q)}({\mathbf{R}}^n),$ where $n \ge 1$ is an integer. If $u$ and $\nabla u$ have absolutely continuous $(p,q)$-norm, then $u \in H_{0}^{1,(p,q)}({\mathbf{R}}^n).$
In particular, $H^{1, (p,q)}({\mathbf{R}}^n)=H_{0}^{1, (p,q)}({\mathbf{R}}^n)$ if $1 \le q<\infty.$
\end{Theorem}

\begin{proof} Let $u \in W^{1,(p,q)}({\mathbf{R}}^n)$ such that $u$ and $\nabla u$ have absolutely continuous $(p,q)$-norm. (This is always the case when $1 \le q<\infty$). Then from Theorem \ref{H=W revisited} it follows that $u$ is in fact in $H^{1,(p,q)}({\mathbf{R}}^n).$

For $j=1,2,\ldots$ choose functions $\varphi_j \in C_{0}^{\infty}(B(0, j+1)),$ $0 \le \varphi_j \le 1,$ such that
$\varphi_j(x)=1$ for each $x \in \overline{B}(0,j).$ Moreover, we choose these functions $\varphi_j$ to be radial and $2$-Lipschitz for all $j \ge 1.$ Then $u_j:=u \varphi_j \in H_{0}^{1,(p,q)}(B(0,j+1))$ for all $j \ge 1$ via Lemma \ref{Product Rule for W1pq}.

Fix $\varepsilon>0.$ For every $j \ge 1$ choose $\psi_j \in C_{0}^{\infty}(B(0,j+1))$ such that
$$||\psi_j-u_j||_{1,(p,q);{\mathbf{R}}^n}=||\psi_j-u \varphi_j||_{1,(p,q);{\mathbf{R}}^n}<\frac{\varepsilon}{4}.$$

For every $j \ge 1$ integer we have, via the definition of the $H^{1,(p,q)}$-norm and via Lemma \ref{Product Rule for W1pq}
\begin{eqnarray*}
||u-u_j||_{1,(p,q); {\mathbf{R}}^n} &\le& ||u-u_j||_{L^{(p,q)}({\mathbf{R}}^n)}+||\nabla u - \nabla u_j||_{L^{(p,q)}({\mathbf{R}}^n; {\mathbf{R}}^n)}\\
&\le& ||u (1-\varphi_j)||_{L^{(p,q)}({\mathbf{R}}^n)}+
||(1-\varphi_j) \nabla u||_{L^{(p,q)}({\mathbf{R}}^n;{\mathbf{R}}^n)}\\
& &+||u \nabla \varphi_j||_{L^{(p,q)}({\mathbf{R}}^n;{\mathbf{R}}^n)}.
\end{eqnarray*}

Since $0 \le \varphi_j \le 1$ is a 2-Lipschitz smooth function supported in $B(0,j+1)$ such that $\varphi_j=1$ in $\overline{B}(0,j),$ this yields
\begin{eqnarray*}
||u-u_j||_{1,(p,q); {\mathbf{R}}^n} &\le& ||u \chi_{{\mathbf{R}}^n \setminus B(0, j)}||_{L^{(p,q)}({\mathbf{R}}^n)}\\
& &+||\nabla u \chi_{{\mathbf{R}}^n \setminus B(0, j)}||_{L^{(p,q)}({\mathbf{R}}^n;{\mathbf{R}}^n)}\\
& &+ || \nabla \varphi_j ||_{L^{\infty}({\mathbf{R}}^n)}  ||u \chi_{{\mathbf{R}}^n \setminus B(0, j)}||_{L^{(p,q)}({\mathbf{R}}^n;{\mathbf{R}}^n)}\\
&\le&3||u \chi_{{\mathbf{R}}^n \setminus B(0, j)}||_{L^{(p,q)}({\mathbf{R}}^n)}\\
& &+||\nabla u \chi_{{\mathbf{R}}^n \setminus B(0, j)}||_{L^{(p,q)}({\mathbf{R}}^n;{\mathbf{R}}^n)}
\end{eqnarray*}
for all $j \ge 1.$

Since $u$ and $\nabla u$ have absolutely continuous $(p,q)$-norm, we can choose an integer $j_0>1$ such that
$$||u \chi_{{\mathbf{R}}^n \setminus B(0, j)}||_{L^{(p,q)}({\mathbf{R}}^n)}+||\nabla u \chi_{{\mathbf{R}}^n \setminus B(0, j)}||_{L^{(p,q)}({\mathbf{R}}^n; {\mathbf{R}}^n)}<\frac{\varepsilon}{4}$$
for all $j \ge j_0.$

Thus, $(\psi_j)_{j \ge 1} \subset C_{0}^{\infty}({\mathbf{R}}^n)$ and
\begin{eqnarray*}
||\psi_j-u||_{1,(p,q); {\mathbf{R}}^n} &\le& ||\psi_j-u_j||_{1,(p,q); {\mathbf{R}}^n}+||u-u_j||_{1,(p,q); {\mathbf{R}}^n}\\
&<&\frac{\varepsilon}{4}+ 3||u \chi_{{\mathbf{R}}^n \setminus B(0, j)}||_{L^{(p,q)}({\mathbf{R}}^n)}\\
&&+||\nabla u \chi_{{\mathbf{R}}^n \setminus B(0, j)}||_{L^{(p,q)}({\mathbf{R}}^n;{\mathbf{R}}^n)}<\varepsilon
\end{eqnarray*}
for all $j \ge j_0.$
This finishes the proof of the theorem.

\end{proof}

We prove now that $W^{1, (p,q_1)}(\Omega) \subsetneq W^{1,(p,q_2)}(\Omega)$ whenever $1<p<\infty$ and $1 \le q_1 < q_2\le \infty.$

\begin{Theorem} \label{W1pr strictly included in W1ps}
Let $n \ge 1$ be an integer and $r>0$ be a positive number. Suppose $1<p<\infty$ and $1 \le q_1< q_2\le\infty.$ Let $\alpha$ be a number in $(0,1]$ such that $1 \le q_1 \le \frac{1}{\alpha}<q_2 \le \infty.$  Let $v_{r, \alpha, n, p}:B(0,r) \rightarrow [0,\infty]$ be the function defined in {\rm(\ref{defn of vralphanp})}.

Then

\par {\rm{(i)}} $v_{r, \alpha, n, p} \in H_{0}^{1, (p, q_2)}(B(0,r)) \setminus H^{1, (p,q_1)}(B(0,r)).$

\par {\rm{(ii)}} $v_{r, \alpha, n, p} \in H^{1, (p, q_2)}(B^{*}(0,r)) \setminus H^{1, (p,q_1)}(B^{*}(0,r)).$

\end{Theorem}

\begin{proof} By choosing $q_3$ such that $\frac{1}{\alpha}<q_3<q_2$ if necessary, we can assume without loss of generality via Theorem \ref{W1pqloc included in H1sloc s<p} (iii) that $q_2<\infty$ throughout the proof of this theorem.

 Since $n, p, r, \alpha, q_1$ and $q_2$ are fixed here, we simplify the notations throughout the proof of the theorem. We let $v_{n, p}:=v_{r, \alpha, n, p}$ and $f_{\alpha, p}:=f_{rad, r, \alpha, p},$ where
 $f_{rad, r, \alpha, p}$ is the function defined in (\ref{defn of fradralphap}).

 Since $$||\nabla v_{n,p}||_{L^{p,q_1}(B^{*}(0,r); {\mathbf{R}}^n)}=\infty,$$ it follows immediately via Theorem \ref{H=W revisited} that $v_{n, p} \notin H^{1, (p, q_1)}(B^{*}(0,r))$ and consequently $v_{n, p} \notin H^{1, (p, q_1)}(B(0,r))=W^{1, (p, q_1)}(B(0,r)).$

 We want to show that $v_{n,p} \in H_{0}^{1,(p,q_2)}(B(0,r)).$ In order to do that, we resort to a truncation argument and we invoke Theorem \ref{HKM93 Thm132}.

 We know from the proof of Theorem \ref{Lpr strictly included in Lps via grad of smooth fns} that $f_{\alpha, p}$ is in $C^{\infty}((0,r)),$ positive and strictly decreasing on $(0,r).$ Moreover, we have $\lim_{t \rightarrow 0} f_{\alpha, p}(t)=\infty$ when $1<p \le n$ and
 $\lim_{t \rightarrow 0} f_{\alpha, p}(t)<\infty$ when $n<p<\infty.$

 For every integer $k \ge 1$ we truncate the function $v_{n, p}$ on the set $B(0, \frac{r}{k+1})$ and we denote this truncation by $v_{n, p, k}.$ Specifically, for $k \ge 1$ integer we define $u_{n, p, k}$ on $B(0,r)$ by
 $$v_{n, p, k}(x)=\left\{\begin{array}{ll}
 f_{\alpha, p}(\frac{r}{k+1}) & \mbox{ if $0 \le |x| \le \frac{r}{k+1}$}\\
 v_{n, p}(x)=f_{\alpha, p}(|x|)  & \mbox{ if $\frac{r}{k+1} < |x| < r.$}
 \end{array}
 \right.
 $$
 It is easy to see that $0 \le v_{n,p,k} \le v_{n,p}$ pointwise in $B(0,r)$ for all $k \ge 1.$ Moreover,  all the functions $v_{n,p,k}$ are Lipschitz on $B(0,r)$ and can be extended continuously by $0$ on $\partial B(0,r).$ More precisely, for all $k \ge 1$ we have
 $$\nabla v_{n, p, k}(x)=\left\{\begin{array}{ll}
 0 & \mbox{ if $0 \le |x| < \frac{r}{k+1}$}\\
 \nabla v_{n, p}(x)=f'_{\alpha, p}(|x|) \frac{x}{|x|}  & \mbox{ if $\frac{r}{k+1} < |x| < r.$}
 \end{array}
 \right.
 $$
 In particular, for every $k \ge 1$ we have $|\nabla v_{n, p, k}| \le |\nabla v_{n,p}|$ almost everywhere in $B(0,r).$ Thus, we have that $(v_{n, p, k})_{k \ge 1} \subset H_{0}^{1,(p, q_2)}(B(0,r)).$ We claim that the sequence $v_{n,p,k}$ is bounded in $H_{0}^{1,(p,q_2)}(B(0,r))$ and in $H^{1,(p,q_2)}(B^{*}(0,r)).$

 We study the cases $n=1$ and $n>1$ separately.

 Case I. We suppose first that $n=1.$ Then $p>n$ and from Theorem \ref{Lpr strictly included in Lps via grad of smooth fns} (iv) it follows that $v_{n,p}$ is continuous and bounded on $B(0,r).$ The boundedness of the sequence $v_{n,p,k}$ in $H_{0}^{1,(p,q_2)}(B(0,r))$ and in $H^{1,(p,q_2)}(B^{*}(0,r))$ is immediate in this case since $0 \le v_{n,p,k} \le v_{n,p}$ pointwise in $B(0,r),$ $|\nabla v_{n,p}| \in L^{(p,q_2)}(B(0,r))$ and since $|\nabla v_{n,p,k}| \le |\nabla v_{n,p}|$ almost everywhere in $B(0,r)$ for every $k \ge 1.$

 Case II. We assume now that $n>1.$ Via Theorem \ref{Sobolev-Poincare for Sobolev-Lorentz} (ii) we have
 \begin{eqnarray*}
 ||v_{n,p,k}||_{L^{p,q_2}(B(0,r))} &\le& C(n,p,q_2) \, |\Omega|^{\frac{1}{n}} ||\nabla v_{n,p,k}||_{L^{p,q_2}(B(0,r); {\mathbf{R}^n})}\\
 &\le& C(n,p,q_2) \, |\Omega|^{\frac{1}{n}} ||\nabla v_{n,p}||_{L^{p,q_2}(B(0,r); {\mathbf{R}^n})}
 \end{eqnarray*}
 for every $k \ge 1$ integer.

 Thus, we proved that the sequence $v_{n,p,k}$ is bounded in $H_{0}^{1,(p,q_2)}(B(0,r))$ and in
 $H^{1,(p,q_2)}(B^{*}(0,r))$ whenever $n \ge 1,$ $1<p<\infty,$  and $1<q_2<\infty.$ The reflexivity of these two spaces and the pointwise convergence of $v_{n,p,k}$ to $v_{n,p}$ on $B^{*}(0,r)$ imply immediately via Theorem \ref{HKM93 Thm132} that $v_{n,p}$ is in fact in $H_{0}^{1,(p,q_2)}(B(0,r))$ and in $H^{1,(p,q_2)}(B^{*}(0,r)).$ Moreover, by invoking Theorem \ref{Holder 1/p' continuity for u in W1pq n equal 1} (i) for $n=1$ and respectively Theorem \ref{Morrey embedding 1<n<p} (iv) for $n>1,$ we see that $v_{n,p}$ is in fact H\"{o}lder continuous in $\overline{B}(0,r)$ with exponent $1-\frac{n}{p}.$
 This finishes the proof.

\end{proof}

\subsection{Chain Rule Results}

We recall the chain rule property for the Sobolev-Lorentz spaces, proved in Costea \cite{Cos3} for $n \ge 2.$

\begin{Theorem} \label{Chain Rule revisited}
{\rm(See Costea \cite[Theorem 3.4.1]{Cos3}).}
Let $\Omega \subset {\mathbf{R}}^n$ be an open set, where $n \ge 1$ is an integer.
Suppose $1<p<\infty$ and $1\le q\le \infty.$
Suppose that $f \in C^{1}({\mathbf{R}}),$ $f(0)=0$ and $f'$ is bounded. If $u \in
W^{1, (p,q)}(\Omega),$ then $f \circ u \in W^{1, (p,q)}(\Omega)$ and
\begin{equation*}
\nabla(f \circ u)=f'(u) \nabla u.
\end{equation*}
Moreover, if $u \in H_{0}^{1, (p,q)}(\Omega),$ then $f \circ u \in
H_{0}^{1, (p,q)}(\Omega).$
\end{Theorem}

\begin{proof} We have
\begin{equation} \label{f circ u dominated by u}
|f \circ u(x)|=|f(u(x))-f(0)| \le ||f'||_{L^{\infty}({\mathbf{R}})} |u(x)| \mbox{ for a.e. $x$ in $\Omega$},
\end{equation}
which implies that $f \circ u \in L^{(p,q)}(\Omega).$

Let $s \in (1,p)$ be fixed. We have that $u \in W^{1,(p,q)}(\Omega),$ hence by Theorem \ref{W1pqloc included in H1sloc s<p} it follows that $u \in H^{1,s}_{loc}(\Omega).$ This and (\ref{f circ u dominated by u}) imply via Ziemer \cite[Theorem 2.1.11]{Zie} that $f \circ u \in H^{1,s}_{loc}(\Omega)$ and
$$\nabla (f \circ u)= f'(u) \nabla u.$$
Thus, we have $f \circ u \in L^{(p,q)}(\Omega),$ $\nabla (f \circ u)=f'(u) \nabla u \in L^{(p,q)}(\Omega; {\mathbf{R}}^n),$ which implies
that $f \circ u \in W^{1,(p,q)}(\Omega).$

We want to prove that $f \circ u \in H_{0}^{1,(p,q)}(\Omega)$ if $u \in H_{0}^{1,(p,q)}(\Omega).$
This was done in Costea \cite[Theorem 3.4.1]{Cos3} in the case $n \ge 2,$ but the proof is valid for $n=1$ as well. We present it for the convenience of the reader.

Suppose that $u \in H_{0}^{1,(p,q)}(\Omega).$ Then it follows immediately that $u$ and $\nabla u$ have absolutely continuous $(p,q)$-norm. From the first part of the proof we already know that $f \circ u \in W^{1,(p,q)}(\Omega)$ because $u \in H_{0}^{1,(p,q)}(\Omega) \subset W^{1,(p,q)}(\Omega).$ Let $u_j, j \ge 1$ be a sequence of functions in $C_{0}^{\infty}(\Omega)$ that converges to $u$ in $H_{0}^{1,(p,q)}(\Omega).$
Without loss of generality, we can assume that $u_j \rightarrow u$ pointwise almost everywhere in $\Omega.$ Since the functions $u_j$ are compactly supported in
$\Omega$ and $f(0)=0,$ it follows that the functions $f \circ u_j$ are compactly supported in $\Omega.$ Moreover, since the functions $u_j$ are in $C^{1}(\Omega)$ and $f$ is in $C^{1}({\mathbf{R}}),$ it follows that the functions $f \circ u_j$ are in $C^{1}(\Omega).$ Thus, $f \circ u_j, j \ge 1$ is a sequence of functions in $C_{0}^{1}(\Omega) \subset H_{0}^{1,(p,q)}(\Omega)$ with $\nabla (f \circ u_j)=f'(u_j) \nabla u_j, j \ge 1.$ Since $f'$ is bounded on ${\mathbf{R}},$ we have

$$|(f \circ u_j)(x)-(f \circ u)(x)| \le ||f'||_{L^{\infty}({\mathbf{R}})} |u_j(x)-u(x)| \mbox{ for a.e. $x$ in } \Omega.$$
This implies that $f \circ u_j$ converges to $f \circ u$ in $L^{(p,q)}(\Omega).$

We have
\begin{eqnarray*}
||f'(u_j) \nabla u_j-f'(u) \nabla u||_{L^{(p,q)}(\Omega; {\mathbf{R}}^n)}&\le& ||f'||_{L^{\infty}({\mathbf{R}})} ||\nabla u_j-\nabla u||_{L^{(p,q)}(\Omega; {\mathbf{R}}^n)}\\
&&+ ||(f'(u_j)-f'(u)) \nabla u||_{L^{(p,q)}(\Omega; {\mathbf{R}}^n)}.
\end{eqnarray*}
The first term of the right-hand side trivially converges to $0.$ The second term of the right-hand side converges to $0$ via Bennett-Sharpley \cite[Proposition I.3.6]{BS} since $\nabla u$ has absolutely continuous $(p,q)$-norm, $f'$ is bounded and $f'(u_j)$ converges to $f'(u)$ pointwise
almost everywhere in $\Omega.$ Consequently, the sequences $f(u_j), f'(u_j) \nabla u_j$ converge to $f(u), f'(u) \nabla u$ respectively and thus $\nabla (f \circ u)=f'(u) \nabla u.$ This finishes the proof.

\end{proof}

Recall the notation
$$u^{+}=\max(u,0) \mbox { and } u^{-}=\min(u,0).$$

\begin{Lemma} \label{pos part revisited}
{\rm(See Costea \cite[Lemma V.12]{Cos0} and \cite[Lemma 3.4.4]{Cos3}).}
Suppose $1<p<\infty$ and $1\le q\le \infty.$
If $u \in W^{1, (p,q)}(\Omega),$ then $u^{+} \in W^{1,
(p,q)}(\Omega)$ and

$$ \nabla u^{+}=\left\{ \begin{array}{cl}
  \nabla u & \mbox{if $u>0$} \\
         0 & \mbox{if $u \le 0.$}
 \end{array}
\right.$$

\end{Lemma}

\begin{proof} Let $s \in (1,p)$ be fixed. From Theorem \ref{W1pqloc included in H1sloc s<p} we have that $u \in H^{1,s}_{loc}(\Omega).$ Via Evans \cite[p.\ 291-292, Exercise 20]{Eva} it follows that
$$ \nabla u^{+}=\left\{ \begin{array}{cl}
  \nabla u & \mbox{if $u>0$} \\
         0 & \mbox{if $u \le 0.$}
 \end{array}
\right.$$
But in that case $\nabla u^{+} \in L^{(p,q)}(\Omega; {\mathbf{R}}^n)$ since $\nabla u$ is in $L^{(p,q)}(\Omega; {\mathbf{R}}^n)$ and $|\nabla u^{+}(x)| \le |\nabla u(x)|$ for almost every $x$ in $\Omega.$ We also have $u^{+} \in L^{(p,q)}(\Omega)$ since $|u^{+}| \le |u|$ in $\Omega$ and $u \in L^{(p,q)}(\Omega).$
So we have in fact that $u^{+} \in W^{1,(p,q)}(\Omega).$ The claim is proved.

\end{proof}

From Theorem \ref{H=W revisited} and Lemma \ref{pos part revisited} it follows immediately that the space $H^{1,(p,q)}(\Omega)$ is closed under truncations from above by nonnegative numbers and
from below by negative numbers whenever $1<p<\infty$ and $1 \le q<\infty.$ Moreover, we have the following density result.

\begin{Theorem} \label{bdd fns in H1pq are dense in H1pq q finite}
{\rm(See Costea \cite[Theorem 3.4.5]{Cos3}).}
Suppose $1<p<\infty$ and $1\le q<\infty.$ Bounded functions in $H^{1,(p,q)}(\Omega)$ are dense in $H^{1,(p,q)}(\Omega).$
\end{Theorem}

It is important to notice that the Sobolev-Lorentz space
$W^{1, (p,q)}(\Omega)$ is a lattice.

\begin{Theorem} \label{W lattice}
{\rm(See Costea \cite[Theorem V.13]{Cos0} and \cite[Theorem 3.4.6]{Cos3}).}
Suppose $1<p<\infty$ and $1\le q \le \infty.$ If $u$ and $v$ are in $W^{1, (p,q)}(\Omega),$
then $\max(u,v)$ and $\min(u,v)$ are in $W^{1, (p,q)}(\Omega)$ with

$$\nabla \max(u,v)(x)=\left\{ \begin{array}{cl}
  \nabla u(x) & \mbox{if $u(x) \ge v(x)$} \\
  \nabla v(x) & \mbox{if $v(x) \ge u(x)$}
 \end{array}
\right.$$

and
$$ \nabla \min(u,v)(x)=\left\{ \begin{array}{cl}
  \nabla u(x) & \mbox{if $u(x) \le v(x)$} \\
  \nabla v(x) & \mbox{if $v(x) \le u(x).$}
 \end{array}
\right.$$

In particular, $|u|=u^{+}-u^{-}$ belongs to $W^{1, (p,q)}(\Omega).$
\end{Theorem}

\begin{Lemma} \label{closed truncation}
{\rm(See Costea \cite[Lemma V.14]{Cos0} and \cite[Lemma 3.4.7]{Cos3}).}
Suppose $1<p<\infty$ and $1\le q<\infty.$ If
$u_j, v_j \in H^{1, (p,q)}(\Omega)$ are such
that $u_j \rightarrow u$ and $v_j \rightarrow v$ in
$H^{1,(p,q)}(\Omega),$ then $\min(u_j, v_j) \rightarrow \min(u,v)$
and similarly $\max(u_j, v_j) \rightarrow \max(u,v)$ in
$H^{1,(p,q)}(\Omega).$
\end{Lemma}

We recall next that the space $H_{0}^{1, (p,q)}(\Omega)$ is also a
lattice whenever $1<p<\infty$ and $1\le q \le \infty.$

\begin{Theorem} \label{H 0 lattice}
{\rm(See Costea \cite[Theorem V.15]{Cos0} and \cite[Theorem 3.4.8]{Cos3}).}
Suppose $1<p<\infty$ and $1\le q\le \infty.$ If $u$ and $v$ are in
$H_{0}^{1, (p,q)}(\Omega),$ then $\max(u,v)$ and $\min(u,v)$ are in
$H_{0}^{1, (p,q)}(\Omega).$ Moreover, if $u \in H_{0}^{1,
(p,q)}(\Omega)$ is nonnegative, then there exists a sequence of
nonnegative functions $\varphi_j \in C_{0}^{\infty}(\Omega)$ that
converges to $u$ in $H_{0}^{1,(p,q)}(\Omega).$
\end{Theorem}

We have a result analogous to Theorem \ref{bdd fns in H1pq are dense in H1pq q finite} for $H_{0}^{1,(p,q)}(\Omega)$ whenever $1<p<\infty$ and $1 \le q \le \infty.$

\begin{Theorem} \label{bdd fns in H01pq are dense in H01pq 1 le q le infty}
{\rm(See Costea \cite[Theorem 3.4.9]{Cos3}).}
Suppose $1<p<\infty$ and $1 \le q \le \infty.$ Bounded functions in $H_{0}^{1,(p,q)}(\Omega)$ are dense in $H_{0}^{1,(p,q)}(\Omega).$
\end{Theorem}

It is easy to see that if a function $u$ is in $H_{0}^{1,(p,q)}(\Omega),$ then $u$ and its distributional gradient $\nabla u$ must have absolutely continuous $(p,q)$-norm. Next we give a sufficient condition for membership in $H_{0}^{1,(p,q)}(\Omega).$

\begin{Lemma} \label{zero on bdry Omega implies membership in H0}
Let $\Omega \subset {\mathbf{R}}^n$ be a bounded open set, where $n \ge 1$ is an integer. Suppose $1<p<\infty$ and $1 \le q \le \infty.$ Suppose that $u$ is a function in $W^{1, (p,q)}(\Omega)$ such that $\lim_{x \rightarrow y} u(x)=0$ for all $y \in \partial \Omega.$ If $\nabla u$ has absolutely continuous $(p,q)$-norm, then $u \in H_{0}^{1,(p,q)}(\Omega).$
\end{Lemma}

\begin{proof} We first show that $u$ has absolutely continuous $(p,q)$-norm if $u$ satisfies the hypotheses of this lemma, a fact which is trivial when $1 \le q<\infty.$ We have to consider the cases $n=1$ and $n \ge 2$ separately.

\vskip 2mm

Case I. We assume first that $n=1.$ Then via Theorem \ref{Holder 1/p' continuity for u in W1pq n equal 1} it follows that $u$ has a version $\overline{u}$ that is H\"{o}lder continuous in $\overline{\Omega}$ with exponent $1-\frac{1}{p}.$ Without loss of generality we can assume that $u$ is H\"{o}lder continuous in $\overline{\Omega}$ with exponent $1-\frac{1}{p}.$ Thus, it follows that $u$ has absolutely continuous $(p,q)$-norm on $\Omega$ if $n=1$ because $u$ is continuous on the bounded set $\overline{\Omega}.$

\vskip 2mm

Case II. We assume now that $n \ge 2.$ Let $s$ be chosen in $(1,p).$ Since $\Omega$ is bounded, it follows via Theorem \ref{W1pqloc included in H1sloc s<p} that $u$ is in $H^{1,s}(\Omega).$ Thus, it follows via Heinonen-Kilpel\"{a}inen-Martio \cite[Lemma 1.26]{HKM} that $u \in H_{0}^{1,s}(\Omega).$
If in addition, $s$ is chosen such that $p<\frac{ns}{n-s},$ then we have via Sobolev's embedding theorem for $H_{0}^{1,s}(\Omega)$ (see Gilbarg-Trudinger \cite[Theorem 7.10]{GT}) that in fact $u \in L^{\frac{ns}{n-s}}(\Omega).$ This, (\ref{relation between Lpr and Lps norm}), Theorem \ref{Holder for Lorentz} and
Bennett-Sharpley \cite[Proposition IV.4.2 and Lemma IV.4.5]{BS} show that actually $u$ has absolutely continuous $(p,q)$-norm for $n \ge 2$ if it satisfies the hypotheses of the lemma.

Since we now know that $u$ and $\nabla u$ have absolutely continuous $(p,q)$-norm whenever $u$ satisfies the hypotheses of the lemma, it follows via Theorem \ref{H=W revisited} that $u$ is in fact in $H^{1,(p,q)}(\Omega).$

By recalling that $u=u^{+}+u^{-},$ it follows immediately via Lemma \ref{pos part revisited} that both $\nabla u^{+}$ and $\nabla u^{-}$ have absolutely
continuous $p,q$-norm since $\nabla u$ has absolutely continuous $p,q$-norm and since $|\nabla u^{+}|, |\nabla u^{-}| \le |\nabla u|$ a.e. in $\Omega.$

We also notice that both $u^{+}$ and $u^{-}$ have absolutely continuous $(p,q)$-norm since $|u^{+}|, |u^{-}| \le |u|$ and since $u$ has absolutely continuous $(p,q)$-norm. Moreover, $\lim_{x \rightarrow y} u^{+}(x)=\lim_{x \rightarrow y} u^{-}(x)=0$ for all $y \in \partial \Omega$ since
$\lim_{x \rightarrow y} u(x)=0$ for all $y \in \partial \Omega.$ Hence, $u^{+}$ and $u^{-}$ satisfy the hypotheses of the lemma if $u$ does, which implies via Theorem \ref{H=W revisited} that $u^{+}$ and $u^{-}$ are in fact in $H^{1,(p,q)}(\Omega).$ Thus, it is enough to prove the claim of the lemma for $u^{+}$ and $u^{-}.$ This means that we can assume without loss of generality that $u \ge 0.$

Fix $\varepsilon>0.$ Let $u_{\varepsilon}=(u-\varepsilon)^{+}=\max(u-\varepsilon, 0).$ Then $u_{\varepsilon}$ has compact support in $\Omega.$
Moreover, via Theorem \ref{W lattice}, we see that $u_{\varepsilon} \in W^{1,(p,q)}(\Omega)$ and
$$ \nabla u_{\varepsilon}=\left\{ \begin{array}{cl}
  \nabla u & \mbox{if $u>\varepsilon$} \\
         0 & \mbox{if $0 \le u \le \varepsilon.$}
 \end{array}
\right.$$

We now show that $u_{\varepsilon}$ is in $H_{0}^{1,(p,q)}(\Omega).$ The function $u_{\varepsilon}$ has absolutely continuous $(p,q)$-norm since $0\le  u_{\varepsilon} \le u$ pointwise in $\Omega$ and since $u$ has absolutely continuous $(p,q)$-norm. Similarly, $\nabla u_{\varepsilon}$ has absolutely continuous $(p,q)$-norm since $0 \le  |\nabla u_{\varepsilon}| \le |\nabla u|$ almost everywhere in $\Omega$ and since $\nabla u$ has absolutely continuous $(p,q)$-norm.
These two facts plus the membership of $u_{\varepsilon}$ in $W^{1,(p,q)}(\Omega)$ yield the membership of $u_{\varepsilon}$ in $H_{0}^{1,(p,q)}(\Omega)$ via Lemma \ref{Product Rule for W1pq}.

We now show that $u_{\varepsilon}$ converges to $u$ in $W^{1,(p,q)}(\Omega).$ Indeed, we see that $0 \le u-u_{\varepsilon} \le \varepsilon$ pointwise on the bounded set $\Omega,$ which implies
$$\lim_{\varepsilon \rightarrow 0} ||u_{\varepsilon}-u||_{L^{(p,q)}(\Omega)}=0.$$

We also see via Theorem \ref{W lattice} and the definition of $u_{\varepsilon}$ that
$$\nabla u-\nabla u_{\varepsilon}=\nabla u \chi_{0<u\le \varepsilon} \mbox{ a.e. in } \Omega.$$
This and the absolute continuity of the $(p,q)$-norm of $\nabla u$ yield
$$\lim_{\varepsilon \rightarrow 0} ||\nabla u_{\varepsilon}-\nabla u||_{L^{(p,q)}(\Omega; {\mathbf{R}^n})}=\lim_{\varepsilon \rightarrow 0} ||\nabla u \chi_{0 < u \le \varepsilon}||_{L^{(p,q)}(\Omega; {\mathbf{R}^n})}=0.$$
This finishes the proof.

\end{proof}

\section{H\"{o}lder continuity of functions in Sobolev-Lorentz spaces} \label{Section Morrey embedding theorems}

In this section we extend some of the known classical embedding theorems to the spaces $H_{0}^{1,(p,q)}(\Omega),$ $C_{0}(\Omega) \cap W^{1,(p,q)}(\Omega)$ and $W_{loc}^{1,(p,q)}(\Omega)$ for $1 \le n<p<\infty$ and $1 \le q \le \infty,$ where $\Omega \subset {\mathbf{R}}^n$ is open. First we recall the definition of H\"{o}lder continuous functions with exponent $0<\alpha<1.$

\begin{Definition} {\rm(See Gilbarg-Trudinger \cite[p.\ 52-53]{GT} and Ziemer \cite[p.\ 2-3]{Zie}).}
Let $n \ge 1$ and $0<\alpha<1.$ Let $u$ be a function defined on a set $D \subset {\mathbf{R}}^n.$ We say that $u$ is \textit{H\"{o}lder continuous in $D$ with exponent $\alpha$} if the quantity
$$[u]_{0, \alpha; D}:=\sup_{x, y \in D, x \neq y} \frac{|u(x)-u(y)|}{|x-y|^{\alpha}}$$
is finite. We say that $u$ is \textit{locally H\"{o}lder continuous in $D$ with exponent $\alpha$}
if $u$ is H\"{o}lder continuous with exponent $\alpha$ on compact subsets of $D.$

Let $\Omega \subset {\mathbf{R}}^n$ be an open set. Let $u$ be a continuous function on $\overline{\Omega}.$ We say that $u$ is in $C^{0, \alpha}(\overline{\Omega})$ if $u$ is H\"{o}lder
continuous in $\Omega$ with exponent $\alpha.$

\end{Definition}

Before we state and prove these embedding results, we need to prove an extension result for functions in $H_{0}^{1,(p,q)}(\Omega).$

\begin{Proposition} \label{Extension by zero in H01pq}
Let $\Omega \subset \widetilde{\Omega}$ be two open sets in ${\mathbf{R}}^n,$ where $n \ge 1$ is an integer. Suppose $1<p<\infty$ and $1\le q \le \infty.$ Let $u$ be a function in $H_{0}^{1,(p,q)}(\Omega)$ and let $\widetilde{u}$ be the extension of $u$ by zero to $\widetilde{\Omega}.$ Then $\widetilde{u} \in H_{0}^{1,(p,q)}(\widetilde{\Omega}).$

\end{Proposition}

\begin{proof} Let $u_k, k \ge 1$ be a sequence of functions in $C_{0}^{\infty}(\Omega)$ such that $u_k$ converges to $u$ in $H_{0}^{1,(p,q)}(\Omega).$ By passing to a subsequence if necessary, we can assume without loss of generality that $u_k$ converges to $u$ pointwise almost everywhere in $\Omega$ and that $\nabla u_k$ converges to $\nabla u$ pointwise almost everywhere in $\Omega.$
For every $k \ge 1$ let $\widetilde{u}_k$ be the extension of $u_k$ by $0$ to $\widetilde{\Omega}.$
Then $\widetilde{u}_k \in C_{0}^{\infty}(\widetilde{\Omega})$ for all $k \ge 1$ and
$$||\widetilde{u}_{k}-\widetilde{u}_{l}||_{H_{0}^{1,(p,q)}(\widetilde{\Omega})}=
||u_{k}-u_{l}||_{H_{0}^{1,(p,q)}(\Omega)}$$
for all $l, k \ge 1.$ Hence the sequence $(\widetilde{u}_{k})_{k \ge 1}$ is fundamental in $H_{0}^{1,(p,q)}(\widetilde{\Omega})$
since the sequence $(u_{k})_{k \ge 1}$ is fundamental in $H_{0}^{1,(p,q)}(\Omega).$ Thus, the sequence $(\widetilde{u}_{k})_{k \ge 1}$ converges in $H_{0}^{1,(p,q)}(\widetilde{\Omega})$ to some function $v \in H_{0}^{1,(p,q)}(\widetilde{\Omega}).$ Since
$\widetilde{u}_k$ converges to $\widetilde{u}$ in $L^{p,q}(\widetilde{\Omega})$ and pointwise almost everywhere in $\widetilde{\Omega},$ it follows in fact that $\widetilde{u}=v$ almost everywhere in $\widetilde{\Omega}.$ Thus, $\widetilde{u} \in H_{0}^{1,(p,q)}(\widetilde{\Omega}).$ This finishes the proof.

\end{proof}

We prove later that if $n=1$ and $\Omega \subset {\mathbf{R}}$ is an open interval, then all the functions in $H_{0}^{1,(p,q)}(\Omega)$ and in $W^{1,(p,q)}(\Omega)$ have representatives that are H\"{o}lder continuous in $\overline{\Omega}$ with exponent $1-\frac{1}{p}.$ The following result is the first step in this direction.

\begin{Proposition}
\label{Holder 1/p' continuity for u in C1R n equal 1}
Suppose $n=1<p<\infty$ and $1 \le q \le \infty$. Let $\Omega \subset {\mathbf{R}}$ be an open interval. If $u \in C^{1}(\Omega)$ and $x, y \in \Omega$ with $x<y,$ then
\begin{equation}
\label{Holder 1/p' continuity estimate for u in C1R n equal 1}
|u(x)-u(y)| \le C(p,q) |x-y|^{1-\frac{1}{p}} ||u'||_{L^{p,q}((x,y))},
\end{equation}
where
\begin{equation}
\label{Cpq is Lp'q' norm of 01}
C(p,q)=\left\{ \begin{array}{ll}
1 & \mbox{ if $q=1$}\\
\left(\frac{p'}{q'}\right)^{\frac{1}{q'}} & \mbox {if $1<q \le \infty.$}
\end{array}
\right.
\end{equation}

\end{Proposition}

\begin{proof} Let $x, y \in \Omega$ with $x<y$ and $u \in C^{1}(\Omega).$
Then $u(y)-u(x)=\int_{x}^{y} u'(t) dt.$ By taking absolute values on both sides and using Theorem \ref{Holder for Lorentz}, we obtain
\begin{eqnarray*}
|u(x)-u(y)| &\le& \int_{x}^{y} |u'(t)| dt = ||u'||_{L^{1}((x,y))} \le ||u'||_{L^{p,q}((x,y))} ||1||_{L^{p',q'}((x,y))}\\
&=& |x-y|^{1-\frac{1}{p}} ||1||_{L^{p',q'}((0,1))} ||u'||_{L^{p,q}((x,y))}.
\end{eqnarray*}
This finishes the proof of the theorem, since $||1||_{L^{p',q'}((0,1))}=C(p,q),$ the constant defined in (\ref{Cpq is Lp'q' norm of 01}).

\end{proof}

\begin{Definition}
We say that the function $\overline{u}$ defined on $\Omega$ is a version of a given function $u$ on $\Omega$ if $u=\overline{u}$ a.e. in $\Omega.$
\end{Definition}

Now we prove (among other things) that if $n=1$ and $\Omega \subset {\mathbf{R}}$ is an open interval, then all the functions in $H_{0}^{1,(p,q)}(\Omega)$ and in $W^{1,(p,q)}(\Omega)$ have representatives that are H\"{o}lder continuous in $\overline{\Omega}$ with exponent $1-\frac{1}{p}.$

\begin{Theorem} \label{Holder 1/p' continuity for u in W1pq n equal 1}
Suppose $n=1<p<\infty$ and $1 \le q \le \infty.$ Let $\Omega \subset {\mathbf{R}}$ be an open set. Let $C(p,q)$ be the constant from {\rm{(\ref{Cpq is Lp'q' norm of 01})}}.

\par {\rm(i)} Suppose that $\Omega$ is an interval. If $u \in W^{1,(p,q)}(\Omega),$ then there exists a version $\overline{u} \in C(\overline{\Omega})$ that is H\"{o}lder continuous in $\overline{\Omega}$ with exponent $1-\frac{1}{p}$ and
$$[\overline{u}]_{0, 1-\frac{1}{p}; \overline{\Omega}} \le C(p,q) ||u'||_{L^{p,q}(\Omega)}.$$

\par {\rm(ii)} If $u \in W_{loc}^{1,(p,q)}(\Omega),$ then there exists a version $\overline{u} \in C(\Omega)$ that is locally H\"{o}lder continuous in $\Omega$ with exponent $1-\frac{1}{p}$ and
$$[\overline{u}]_{0, 1-\frac{1}{p}; \overline{\Omega'}} \le C(p,q) ||u'||_{L^{p,q}(\Omega')},$$
whenever $\Omega'$ is an open subinterval of $\Omega$ such that $\Omega' \subset \subset \Omega.$
Moreover, if $u' \in L^{(p,q)}(\Omega)$ and $\Omega$ is an interval, then $\overline{u}$ is H\"{o}lder continuous in $\overline{\Omega}$ with exponent $1-\frac{1}{p}$ and
$$[\overline{u}]_{0, 1-\frac{1}{p}; \overline{\Omega}} \le C(p,q) ||u'||_{L^{p,q}(\Omega)}.$$

\par {\rm(iii)} If $u \in H_{0}^{1,(p,q)}(\Omega),$ then there exists a version $\overline{u} \in C(\overline{\Omega})$ that is H\"{o}lder continuous in $\overline{\Omega}$ with exponent $1-\frac{1}{p}$ and
$$[\overline{u}]_{0, 1-\frac{1}{p}; \overline{\Omega}} \le C(p,q) ||u'||_{L^{p,q}(\Omega)}.$$

\end{Theorem}

\begin{proof} From Theorem \ref{W1pqloc included in H1sloc s<p} we have $W_{loc}^{1,(p,q)}(\Omega) \subset H_{loc}^{1,s}(\Omega)$ for every $1<s<p.$ Hence, it follows via Evans \cite[p.\ 290, Exercise 6]{Eva} that $u$ has a version $\overline{u} \in C(\Omega)$ that is locally H\"{o}lder continuous in
$\Omega$ with exponent $1-\frac{1}{s}.$ Without loss of generality we can assume that $u$ is itself continuous in $\Omega \subset {\mathbf{R}}.$

For both (i) and (ii) we prove first that
$$|u(x)-u(y)| \le C(p,q) |x-y|^{1-\frac{1}{p}} ||u'||_{L^{p,q}((x,y))}$$
whenever $x$ and $y$ are two points in $\Omega$ such that $x<y$ and $(x,y) \subset \Omega.$ Here $C(p,q)$ is the constant from (\ref{Cpq is Lp'q' norm of 01}).

The function $u$ is assumed to be continuous in $\Omega$ and the above pointwise inequality is local; thus, in order to prove it, we can assume without loss of generality for both (i) and (ii) that $\Omega \subset {\mathbf{R}}$ is a bounded open interval and that $u$ is compactly supported in $\Omega.$

From Theorem \ref{W1pqloc included in H1sloc s<p} it follows (since $\Omega$ is assumed to be bounded) that $u \in H_{0}^{1,s}(\Omega)$ for every $s \in (1,p).$ Fix such an $s \in (1,p).$ Let $u_k, k \ge 1$ be a sequence in $C_{0}^{\infty}(\Omega)$ converging to $u \in H_{0}^{1,s}(\Omega).$

Since $H_{0}^{1,(p,q)}(\Omega) \subset H_{0}^{1,s}(\Omega),$ $u$ is continuous in $\Omega$ and $n=1<p<\infty,$ it follows immediately that $u_k$ in fact converges to $u$ uniformly on $\Omega.$

The pointwise and uniform convergence of $u_k$ to $u$ on $\Omega,$ the fact that $u_k'$ converges to $u'$ in $L^{(p,q)}(\Omega)$ and the fact that (\ref{Holder 1/p' continuity estimate for u in C1R n equal 1}) holds for every $k \ge 1$ and for all $x,y \in \Omega$ with $x<y$ imply immediately by passing to the limit in
(\ref{Holder 1/p' continuity estimate for u in C1R n equal 1}) that
\begin{eqnarray*}
|u(x)-u(y)| &\le& \int_{x}^{y} |u'(t)| dt = ||u'||_{L^{1}((x,y))} \\
&\le& C(p,q) |x-y|^{1-\frac{1}{p}} ||u'||_{L^{p,q}((x,y))} \\
&\le& C(p,q) |x-y|^{1-\frac{1}{p}} ||u'||_{L^{p,q}(\Omega)}
\end{eqnarray*}
for all $x, y \in \Omega$ with $x<y.$ Here $C(p,q)$ is the constant
from (\ref{Cpq is Lp'q' norm of 01}).
This finishes the proof of the desired pointwise inequality.

\vskip 2mm

Now we prove claim (i). Since $u' \in L^{(p,q)}(\Omega),$ the above pointwise inequality implies that
$$|u(x)-u(y)| \le C(p,q) |x-y|^{1-\frac{1}{p}} ||u'||_{L^{p,q}(\Omega)}$$
for all $x, y \in \Omega$ with $x<y.$ In particular, $u$ is uniformly continuous on $\Omega.$

We claim that $u$ admits a continuous extension to $\overline{\Omega}.$ This is obvious when $\Omega={\mathbf{R}},$ so we can assume without loss of generality that $\Omega \neq {\mathbf{R}}.$

If $\Omega=(a,b)$ is a bounded interval, then it follows that $u$ is bounded on $\Omega$ and uniformly continuous on $\Omega,$ so in this case we can indeed extend $u$ continuously to $\overline{\Omega}=[a,b].$ We denote the extension to $\overline{\Omega}=[a,b]$ by $u$ as well.

If $\Omega \neq {\mathbf{R}}$ is an unbounded interval, then $\Omega$ is either $(a, \infty)$ or $(-\infty, a)$ for some $a \in {\mathbf{R}}.$ In both situations, $u$ is uniformly continuous on $\Omega$ and bounded near $x=a,$ so in this case we can also extend $u$ continuously to the unbounded set $\overline{\Omega}=\Omega \cup \{a\}.$ We denote this continuous extension to $\overline{\Omega}$ by $u$ as well.

The above pointwise inequality and the continuity of $u$ on $\overline{\Omega}$ imply that
$$|u(x)-u(y)| \le C(p,q) |x-y|^{1-\frac{1}{p}} ||u'||_{L^{p,q}(\Omega)}$$
for all $x, y \in \overline{\Omega}$ with $x<y.$ This finishes the proof of claim (i).

\vskip 2mm

Now we prove claim (ii). The first part of claim (ii) follows immediately from (i).

Assume now that $u \in W_{loc}^{1,(p,q)}(\Omega),$ $\Omega$ is an interval and $u' \in L^{(p,q)}(\Omega).$ By mimicking the argument from the proof of claim (i), we see that $u$ admits a continuous extension to $\overline{\Omega}.$ If we denote that extension by $u$ as well, we see that
$$|u(x)-u(y)| \le C(p,q) |x-y|^{1-\frac{1}{p}} ||u'||_{L^{p,q}(\Omega)}$$
for all $x, y \in \overline{\Omega}$ with $x<y.$ This finishes the proof of claim (ii).

Now we prove claim (iii). If $\Omega \subset {\mathbf{R}}$ is an interval, then claim (iii) follows obviously from (i).

Suppose now that $\Omega$ is not an interval. Let $U$ be the smallest open interval containing $\Omega.$ That is, $U=(a,b),$ where
$$a=\inf_{x \in \Omega} x \mbox{ and } b=\sup_{x \in \Omega} x.$$
Let $\widetilde{u}$ be the extension of $u$ by $0$ to $U.$ Then $\widetilde{u} \in H_{0}^{1,(p,q)}(U)$ via Proposition \ref{Extension by zero in H01pq}. Thus, claim (iii) holds for $\widetilde{u} \in H_{0}^{1,(p,q)}(U).$

We notice that $||\widetilde{u}'||_{L^{p,q}(U)}=||u'||_{L^{p,q}(\Omega)}.$ Moreover, it is easy to see that $u$ is in $C^{0,1-\frac{1}{p}}(\overline{\Omega})$ if and only if $\widetilde{u}$ is in $C^{0,1-\frac{1}{p}}(\overline{U})$ with
$$[\widetilde{u}]_{0, 1-\frac{1}{p}; \overline{U}}=[u]_{0, 1-\frac{1}{p}; \overline{\Omega}}.$$
Thus, claim (iii) holds also for $u \in H_{0}^{1,(p,q)}(\Omega)$ when $\Omega$ is not an interval. This finishes the proof of the theorem.

\end{proof}

Now we prove (among other things) that if $1<n<p<\infty$ and $1 \le q \le \infty,$ then the spaces $H_{0}^{1,(p,q)}(\Omega)$ and $C_{0}(\Omega) \cap W^{1,(p,q)}(\Omega)$ embed into $C^{0, 1-\frac{n}{p}}(\overline{\Omega}).$ Since we work with functions in $H_{0}^{1,(p,q)}(\Omega)$ and in $C_{0}(\Omega) \cap W^{1,(p,q)}(\Omega),$ no regularity assumptions on $\partial \Omega$ are needed.
This extends the Morrey embedding theorem to the Sobolev-Lorentz spaces in the Euclidean setting. We prove this theorem by relying on the well-known Poincar\'{e} inequality in the Euclidean setting and by invoking the classical Morrey embedding theorem for $1<n<s<p<\infty,$ proved by Evans in \cite{Eva} and by Gilbarg-Trudinger in \cite{GT}. Theorem \ref{Morrey embedding 1<n<p} (i) was also obtained via a different proof by Cianchi-Pick in \cite{CiPi2}. (See Cianchi-Pick \cite[Theorem 1.3]{CiPi2}).

\begin{Theorem} \label{Morrey embedding 1<n<p}
Suppose $1<n<p<\infty$ and $1 \le q \le \infty.$ Let $\Omega \subset {\mathbf{R}}^n$ be open.

\par {\rm(i)} If $u \in W^{1,(p,q)}(\Omega)$ is compactly supported in $\Omega,$
then $u$ has a version $\overline{u} \in C^{0,1-\frac{n}{p}}(\overline{\Omega})$ and
\begin{equation} \label{Holder 1-n/p estimate Morrey embedding 1<n<p}
[\overline{u}]_{0, 1-\frac{n}{p}; \overline{\Omega}} \le C(n,p,q) ||\nabla u||_{L^{p,q}(\Omega; {\mathbf{R}}^n)},
\end{equation}
where $C(n,p,q)>0$ is a constant that depends only on $n,p,q.$

\par {\rm(ii)} If $u \in W^{1,(p,q)}_{loc}(\Omega),$ then $u$ has a version $\overline{u}$ that is locally H\"{o}lder continuous in $\Omega$ with exponent $1-\frac{n}{p}.$

\par {\rm(iii)} If $u \in W^{1,(p,q)}_{loc}({\mathbf{R}}^n)$ and $|\nabla u| \in L^{(p,q)}({\mathbf{R}}^n),$
then $u$ has a version $\overline{u} \in C^{0,1-\frac{n}{p}}({\mathbf{R}}^n)$ and
$$[\overline{u}]_{0, 1-\frac{n}{p}; {\mathbf{R}}^n} \le C(n,p,q) ||\nabla u||_{L^{p,q}({\mathbf{R}}^n; {\mathbf{R}}^n)},$$
where $C(n,p,q)$ is the constant from {\rm(\ref{Holder 1-n/p estimate Morrey embedding 1<n<p})}.

\par {\rm(iv)} If $u \in H_{0}^{1,(p,q)}(\Omega),$ then $u$ has a version $\overline{u} \in C^{0,1-\frac{n}{p}}(\overline{\Omega})$ and
$$[\overline{u}]_{0, 1-\frac{n}{p}; \overline{\Omega}} \le C(n,p,q) ||\nabla u||_{L^{p,q}(\Omega; {\mathbf{R}}^n)},$$
where $C(n,p,q)$ is the constant from {\rm(\ref{Holder 1-n/p estimate Morrey embedding 1<n<p})}.

\end{Theorem}

\begin{proof} Let $s \in (n,p)$ be fixed. We have via Gilbarg-Trudinger \cite[Theorem 7.17]{GT} and via Theorem \ref{W1pqloc included in H1sloc s<p} that $W^{1,(p,q)}_{loc}(\Omega)$ embeds into the space of locally H\"{o}lder continuous functions in $\Omega$ with exponent $1-\frac{n}{s}.$ Thus, if $u \in W^{1,(p,q)}_{loc}(\Omega)$ with $1<n<p<\infty,$ we can assume without loss of generality throughout the proof of this theorem (after possibly redefining $u$ on a subset of $\Omega$ of Lebesgue measure $0$) that $u$ is in fact locally H\"{o}lder continuous in $\Omega$ with exponent $1-\frac{n}{s}.$

We prove now (i). Suppose that $u \in W^{1,(p,q)}(\Omega)$ is compactly supported in $\Omega.$
Then we can assume without loss of generality that $\Omega$ is bounded. Since $u$ is compactly supported in $\Omega$ and $u$ is locally H\"{o}lder continuous in $\Omega$ with exponent $1-\frac{n}{s},$ it follows in fact that $u$ can be extended continuously by $0$ on $\partial \Omega$ and this extension (denoted by $u$ as well) is in fact H\"{o}lder continuous in $\overline{\Omega}$ with exponent $1-\frac{n}{s},$ where $1<n<s<p.$

We extend $u$ by $0$ to ${\mathbf{R}}^n \setminus \Omega$ and we denote this extension by $v.$ Since $u \in C_{0}(\Omega) \cap W^{1,(p,q)}(\Omega),$ it follows immediately from the definition of $v$ that $v \in C_{0}({\mathbf{R}}^n) \cap W^{1,(p,q)}({\mathbf{R}}^n)$ and
$$\nabla v(x)=\left\{\begin{array}{ll}
\nabla u(x) & \mbox{if $x \in \Omega$}\\
0 & \mbox{if $x \in {\mathbf{R}}^n \setminus \Omega$}.
\end{array}
\right.
$$
Moreover, since $u \in C^{0, 1-\frac{n}{s}}(\overline{\Omega}),$ it is easy to see that $v \in C^{0, 1-\frac{n}{s}}({\mathbf{R}}^n)$ with
$$[v]_{0, 1-\frac{n}{s}; {\mathbf{R}}^n}=[u]_{0, 1-\frac{n}{s}; \overline{\Omega}}.$$
It is also easy to see that $v \in C^{0, 1-\frac{n}{p}}({\mathbf{R}}^n)$ if and only if $u \in C^{0, 1-\frac{n}{p}}(\overline{\Omega})$ with
$$[v]_{0, 1-\frac{n}{p}; {\mathbf{R}}^n}=[u]_{0, 1-\frac{n}{p}; \overline{\Omega}}.$$

It is enough to show that $$[v]_{0, 1-\frac{n}{p}; {\mathbf{R}}^n} \le C(n,p,q) ||\nabla v||_{L^{p,q}({\mathbf{R}}^n;{\mathbf{R}}^n)}.$$

Let $x \neq y$ be two points from ${\mathbf{R}}^n$ and let $a$ be the midpoint of the segment connecting $x$ and $y.$ Let $R=|x-y|.$

For every integer $j \ge 0$ let $B_{x,j}=B(x, 2^{-j-1}R)$ and $B_{y,j}=B(y, 2^{-j-1}R).$ Let $B_{a}=B(a,R).$ It is easy to see that $B_{x,0} \cup B_{y,0} \subset B_a.$

Since $v$ is continuous in ${\mathbf{R}}^n,$ all the points in ${\mathbf{R}}^n$ are Lebesgue points for $v.$ Thus,
$$v(x)=\lim_{j \rightarrow \infty} v_{B_{x,j}} \mbox{ and } v(y)=\lim_{j \rightarrow \infty} v_{B_{y,j}}.$$
Hence
\begin{eqnarray*}
v(x)-v(y)&=&\left((v_{B_{x,0}}-v_{B_{a}})+\sum_{j=1}^{\infty} \frac{1}{|B_{x,j+1}|} \int_{B_{x, j+1}} (v(z)-v_{B_{x,j}}) \, dz\right)\\
& &-\left((v_{B_{y,0}}-v_{B_{a}})+\sum_{j=1}^{\infty} \frac{1}{|B_{y,j+1}|} \int_{B_{y, j+1}} (v(z)-v_{B_{y,j}}) \, dz\right).
\end{eqnarray*}
This implies
\begin{eqnarray*}
|v(x)-v(y)|&\le&\left(|v_{B_{x,0}}-v_{B_{a}}|+\sum_{j=1}^{\infty} \frac{1}{|B_{x,j+1}|} \int_{B_{x, j+1}} |v(z)-v_{B_{x,j}}| \, dz\right)\\
& &+\left(|v_{B_{y,0}}-v_{B_{a}}|+\sum_{j=1}^{\infty} \frac{1}{|B_{y,j+1}|} \int_{B_{y, j+1}} |v(z)-v_{B_{y,j}}| \, dz\right).
\end{eqnarray*}
Since $v \in W^{1,(p,q)}({\mathbf{R}}^n)$ is compactly supported in ${\mathbf{R}}^n,$ then via Theorem \ref{W1pqloc included in H1sloc s<p} we have $v \in H_{0}^{1,s}({\mathbf{R}}^n).$ Thus, via Poincar\'{e}'s inequality, we have
\begin{equation} \label{Poincare 11 inequality Rn}
\frac{1}{B(w,r)} \int_{B(w,r)} |v(z)-v_{B(w,r)}| dz \le C(n) r \frac{1}{B(w,r)} \int_{B(w,r)} |\nabla v(z)| dz
\end{equation}
for every $w \in {\mathbf{R}}^n$ and every $r>0,$ where $C(n)>0$ is a constant that depends only on $n.$

Since $B_{x,0} \subset B_{a},$ we have via H\"{o}lder's inequality for Lorentz spaces (see Theorem \ref{Holder for Lorentz}) and Poincar\'{e}'s inequality (\ref{Poincare 11 inequality Rn})
\begin{eqnarray*}
|v_{B_{x,0}}-v_{B_a}|&=&\frac{1}{|B_{x,0}|} \left| \int_{B_{x,0}} (v(z)-v_{B_a}) \, dz \right|\\
&\le&\frac{1}{|B_{x,0}|} \int_{B_{x,0}} |v(z)-v_{B_a}| \, dz\\
&\le& \frac{2^n}{|B_a|} \int_{B_{a}} |v(z)-v_{B_a}| \, dz\\
&\le& C(n) R \frac{1}{|B_a|} \int_{B_{a}} |\nabla v(z)| \, dz \\
&\le& C(n,p,q) R \left(\frac{||\nabla v||_{L^{p,q}(B_a; {\mathbf{R}}^n)}^p}{|B_a|}\right)^{1/p} \\
&=& C(n,p,q) R^{1-\frac{n}{p}} ||\nabla v||_{L^{p,q}(B_a; {\mathbf{R}}^n)}\\
&=& C(n,p,q) R^{1-\frac{n}{p}} ||\nabla u||_{L^{p,q}(\Omega \cap B_a; {\mathbf{R}}^n)}\\
&\le&C(n,p,q) R^{1-\frac{n}{p}} ||\nabla u||_{L^{p,q}(\Omega; {\mathbf{R}}^n)}\\
&=& C(n,p,q) R^{1-\frac{n}{p}} ||\nabla v||_{L^{p,q}({\mathbf{R}}^n; {\mathbf{R}}^n)}.
\end{eqnarray*}

Similarly, since $B_{y,0} \subset B_{a},$ we obtain (after an almost identical reasoning, by replacing $B_{x,0}$ with $B_{y,0}$)
\begin{eqnarray*}
|v_{B_{y,0}}-v_{B_a}| &\le& C(n,p,q) R^{1-\frac{n}{p}} ||\nabla u||_{L^{p,q}(\Omega \cap B_a; {\mathbf{R}}^n)}\\
&\le& C(n,p,q) R^{1-\frac{n}{p}} ||\nabla u||_{L^{p,q}(\Omega; {\mathbf{R}}^n)}.
\end{eqnarray*}

We want to obtain upper estimates for
\begin{eqnarray*}
|u_{B_{x,j+1}}-u_{B_{x,j}}| &=& \frac{1}{|B_{x,j+1}|} \left| \int_{B_{x,j+1}} (u(z)-u_{B_{x,j}}) \, dz  \right| \mbox{ and } \\
|u_{B_{y,j+1}}-u_{B_{y,j}}| &=& \frac{1}{|B_{y,j+1}|} \left| \int_{B_{y,j+1}} (u(z)-u_{B_{y,j}}) \, dz  \right|
\end{eqnarray*}
for all $j \ge 0.$

For all $j \ge 0$ we only do the estimate for $|u_{B_{x,j+1}}-u_{B_{x,j}}|$ because we would use an almost identical reasoning to obtain the estimate for $|u_{B_{x,j+1}}-u_{B_{x,j}}|.$

Since $B_{x,j+1} \subset B_{x,j} \subset B_a$ for all $j \ge 0,$ we have via H\"{o}lder's inequality for Lorentz spaces (see Theorem \ref{Holder for Lorentz}) and Poincar\'{e}'s inequality (\ref{Poincare 11 inequality Rn})
\begin{eqnarray*}
|v_{B_{x,j+1}}-v_{B_{x,j}}|&=&\frac{1}{|B_{x,j+1}|} \left| \int_{B_{x,j+1}} (v(z)-v_{B_{x,j}}) \, dz \right|\\
&\le& \frac{1}{|B_{x,j+1}|} \int_{B_{x,j+1}} |v(z)-v_{B_{x,j}}| \, dz\\
&\le& \frac{2^n}{|B_{x,j}|} \int_{B_{x,j}} |v(z)-v_{B_{x,j}}| \, dz\\
&\le& C(n) R \frac{1}{|B_{x,j}|} \int_{B_{x,j}} |\nabla v(z)| \, dz \\
&\le& C(n,p,q) 2^{-j}R \left(\frac{||\nabla v||_{L^{p,q}(B_{x,j}; {\mathbf{R}}^n)}^p}{|B_{x,j}|}\right)^{1/p} \\
&\le& C(n,p,q) 2^{-j}R \left(\frac{||\nabla v||_{L^{p,q}(B_a; {\mathbf{R}}^n)}^p}{|B_{x,j}|}\right)^{1/p} \\
&=& C(n,p,q) (2^{-j}R)^{1-\frac{n}{p}} ||\nabla v||_{L^{p,q}(B_a; {\mathbf{R}}^n)}\\
&\le& C(n,p,q) (2^{-j}R)^{1-\frac{n}{p}} ||\nabla v||_{L^{p,q}({\mathbf{R}}^n; {\mathbf{R}}^n)}.
\end{eqnarray*}

By summing the above inequalities and taking into account that $|x-y|=R,$ we have
\begin{eqnarray*}
|v(x)-v(y)| &\le& C(n,p,q) ||\nabla v||_{L^{p,q}(B_a; {\mathbf{R}}^n)} \left(\sum_{j=0}^{\infty} (2^{-j}R)^{1-\frac{n}{p}}\right)\\
&=& C(n,p,q) R^{1-\frac{n}{p}} ||\nabla v||_{L^{p,q}(B_a; {\mathbf{R}}^n)}\\
&=& C(n,p,q) |x-y|^{1-\frac{n}{p}} ||\nabla v||_{L^{p,q}(B_a; {\mathbf{R}}^n)}\\
&\le& C(n,p,q) |x-y|^{1-\frac{n}{p}} ||\nabla v||_{L^{p,q}({\mathbf{R}}^n; {\mathbf{R}}^n)}.
\end{eqnarray*}

Claim (i) holds with constant $C(n,p,q)$ from the last line in the above sequence of inequalities. This finishes the proof of claim (i).

 \vskip 2mm

 We prove now claim (ii). Let $\Omega' \subset \subset \Omega$ be an open subset of $\Omega$ and let $u \in W^{1,(p,q)}_{loc}(\Omega).$ We choose a cut-off function $\varphi \in C_{0}^{\infty}(\Omega)$ such that $0 \le \varphi \le 1$ and such that $\varphi=1$ in $\overline{\Omega'}.$ Then $u \varphi$ is compactly supported in
 $\Omega$ and via Lemma \ref{Product Rule for W1pq} we have $u \varphi \in W^{1,(p,q)}(\Omega).$
 From part (i) we have that $u \varphi$ is H\"{o}lder continuous in $\overline{\Omega}$ with exponent $1-\frac{n}{p}$ and
 $$[u\varphi]_{0, 1-\frac{n}{p}; \overline{\Omega}} \le C(n,p,q) ||\nabla(u \varphi)||_{L^{p,q}(\Omega; {\mathbf{R}}^n)},$$
 where $C(n,p,q)>0$ is the constant from (\ref{Holder 1-n/p estimate Morrey embedding 1<n<p}). Since $\Omega' \subset \subset \Omega$ and $\varphi=1$ in $\overline{\Omega'},$ it follows that $u \varphi=u$ in $\overline{\Omega'}.$ Thus, $u$ is H\"{o}lder continuous in $\overline{\Omega'}$ with exponent $1-\frac{n}{p}$ and
 \begin{eqnarray*}
 [u]_{0, 1-\frac{n}{p}; \overline{\Omega'}}&=&[u \varphi]_{0, 1-\frac{n}{p}; \overline{\Omega'}}
 \le [u \varphi]_{0, 1-\frac{n}{p}; \overline{\Omega}}\\
 &\le& C(n,p,q) ||\nabla(u \varphi)||_{L^{p,q}(\Omega; {\mathbf{R}}^n)},
 \end{eqnarray*}
 where $C(n,p,q)>0$ is the constant from (\ref{Holder 1-n/p estimate Morrey embedding 1<n<p}).
 This finishes the proof of (ii).

 \vskip 2mm

 We prove now claim (iii). We use the notation from part (i).

 Let $x \neq y$ be two points in ${\mathbf{R}}^n,$ let $a$ be the midpoint of the segment $[x,y]$ and let $R=|x-y|.$ Let $\varphi_{x,y} \in C_{0}^{\infty}({\mathbf{R}}^n)$ be a function such that $0 \le \varphi_{x,y} \le 1$ and such that $\varphi_{x,y}=1$ on $B_a:=B(a,R).$

 By running the argument from (i) with the function $u \varphi_{x,y} \in W^{1,(p,q)}(\Omega)$ that is compactly supported in ${\mathbf{R}}^n$ we obtain, since $\varphi_{x,y}=1$ on $B_{a} \ni \{x,y \}$
 \begin{eqnarray*}
 |u(x)-u(y)|&=&|(u \varphi_{x,y})(x)-(u \varphi_{x,y})(y)|\\
 &\le& C(n,p,q) ||\nabla (u \varphi_{x,y})||_{L^{p,q}(B_a; {\mathbf{R}}^n)}\\
 &=&C(n,p,q) ||\nabla u||_{L^{p,q}(B_a; {\mathbf{R}}^n)} \le C(n,p,q) ||\nabla u||_{L^{p,q}({\mathbf{R}}^n; {\mathbf{R}}^n)},
 \end{eqnarray*}
 where $C(n,p,q)$ is the constant from (\ref{Holder 1-n/p estimate Morrey embedding 1<n<p}) and from the last line of the last sequence of inequalities in the proof of claim (i). This finishes the proof of claim (iii).

 \vskip 2mm

 We prove now claim (iv). We have to consider the cases $\Omega={\mathbf{R}}^n$ and
 $\Omega \subsetneq {\mathbf{R}}^n$ separately.

 Suppose first that $\Omega={\mathbf{R}}^n.$ In this case the claim follows immediately from (iii)
 because the membership of $u$ in $H_{0}^{1, (p,q)}({\mathbf{R}}^n)$ implies that $u$ is in $W_{loc}^{1, (p,q)}({\mathbf{R}}^n)$ and $|\nabla u| \in L^{(p,q)}({\mathbf{R}}^n).$

 Suppose now that $\Omega \subsetneq {\mathbf{R}}^n.$ Let $v$ be the extension by $0$ of $u$ to ${\mathbf{R}}^n \setminus \Omega.$ We claim that $v$ is continuous in ${\mathbf{R}}^n.$

 Indeed, let $(u_k)_{k \ge 1} \subset C_{0}^{\infty}(\Omega)$ be a sequence of functions such that
 $u_k$ converges to $u$ in $H_{0}^{1,(p,q)}(\Omega)$ and pointwise almost everywhere in $\Omega.$ For every $k \ge 1$ let $v_k$ be the extension by $0$ of $u_k$ to ${\mathbf{R}}^n \setminus \Omega.$ We see immediately that $(v_k)_{k \ge 1} \subset C_{0}^{\infty}({\mathbf{R}}^n)$ and that $v_k$ converges to $v$ pointwise almost everywhere in ${\mathbf{R}}^n.$ Moreover, from Proposition \ref{Extension by zero in H01pq}, it follows that $v_k$ converges to $v$ in $H_{0}^{1,(p,q)}({\mathbf{R}^n}).$

 Since the sequence $(u_k)_{k \ge 1} \subset C_{0}^{\infty}(\Omega)$ converges to $u$ in $H_{0}^{1,(p,q)}(\Omega),$ $u$ is continuous in $\Omega,$ and $1<n<p<\infty,$ it follows immediately that in fact $u_k$ converges to $u$ uniformly on compact subsets of $\Omega.$ In particular, $u_k$ converges pointwise to $u$ everywhere in $\Omega.$

 From this, the definition of $v$ and of the functions $v_k$ and from the fact that $v_k=v=0$ everywhere on ${\mathbf{R}}^n \setminus \Omega$ for all $k \ge 1,$ it follows that the sequence $v_k$ converges pointwise to $v$ everywhere in ${\mathbf{R}}^n.$ Since the sequence $(v_k)_{k \ge 1} \subset C_{0}^{\infty}({\mathbf{R}}^n)$ converges to $v$ in $H_{0}^{1,(p,q)}({\mathbf{R}}^n)$ and pointwise in
 ${\mathbf{R}}^n$ and since $1<n<p<\infty,$ it follows immediately that in fact $v \in C({\mathbf{R}}^n)$ and the sequence $v_k$ converges to $v$ uniformly on compact subsets of ${\mathbf{R}}^n.$ Thus, we proved that $v$ is continuous in ${\mathbf{R}}^n.$ If we denote the extension of $u$ by $0$ to $\partial \Omega$ by $u$ as well, the above argument proved that $u \in C(\overline{\Omega}).$

 We see now that the claim (iv) holds for $v$ via (iii). Since $v \in C^{0, 1-\frac{n}{p}}({\mathbf{R}}^n),$ and $v$ is the continuous extension by $0$ of the function $u \in C(\overline{\Omega})$ to ${\mathbf{R}}^n,$ it follows immediately that $u \in C^{0, 1-\frac{n}{p}}(\overline{\Omega})$ and
 $$[u]_{0, 1-\frac{n}{p}; \overline{\Omega}}=[v]_{0, 1-\frac{n}{p}; {\mathbf{R}}^n}.$$
 This implies immediately that the claim (iv) holds for $u$ as well, because
 $$||\nabla u||_{L^{p,q}(\Omega; {\mathbf{R}}^n)}=||\nabla v||_{L^{p,q}({\mathbf{R}}^n; {\mathbf{R}}^n)}.$$
 This finishes the proof of the theorem.

\end{proof}

 Theorems \ref{Holder 1/p' continuity for u in W1pq n equal 1} and \ref{Morrey embedding 1<n<p} together with Proposition \ref{up in Lpinfty loc and in W1pinfty loc minus H1pinfty} yield the following corollary.

 \begin{Corollary}
 Suppose $1 \le n<p<\infty,$ where $n$ is an integer. Let
 $u_p: {\mathbf{R}}^n \rightarrow {\mathbf{R}},$ $u_p(x)=|x|^{1-\frac{n}{p}}.$
 Then $u_p$ is H\"{o}lder continuous in ${\mathbf{R}}^n$ with exponent $1-\frac{n}{p}.$
 \end{Corollary}

 \begin{proof} We proved in Proposition \ref{up in Lpinfty loc and in W1pinfty loc minus H1pinfty}
 that $u_p \in W_{loc}^{1,(p,\infty)}({\mathbf{R}}^n)$ for all $p \in (1, \infty).$ The claim follows immediately by invoking Theorem \ref{Holder 1/p' continuity for u in W1pq n equal 1} (ii) for $n=1$ and respectively Theorem \ref{Morrey embedding 1<n<p} (iii) for $n>1.$
 One can also see that the claim holds via a direct and easy computation.
 \end{proof}

 \vspace{2mm}

 \noindent {\bf{Acknowledgements.}} I started writing this article towards the end of my stay
 at the University of Trento and I finished it during my stay at the University of Pisa.
 I would like to thank Professor Giuseppe Buttazzo for his helpful suggestions.
 I would also like to thank the referees for their comments and for helping me improve the paper.

\end{document}